\documentclass{article}

\usepackage{amsmath,amssymb,amsthm}
\usepackage{times}
\usepackage{graphicx}
\usepackage{color}
\usepackage{multirow}
\usepackage{rotating}
\usepackage{bbm}
\usepackage{latexsym}

\usepackage[english]{babel}
\usepackage[lofdepth,lotdepth]{subfig}
\usepackage{enumitem} 
\usepackage{tikz}
\usepackage{stmaryrd}
\usepackage{hyperref}

\hypersetup{pdfstartview=XYZ}

\newcommand{\e}{{\mathrm{e}}}
\newcommand{\Bin}{{\mathrm{Bin}}} 
\newcommand{\BinMix}{{\mathrm{BinMix}}}
\newcommand{\Bino}[1]{\mathrm{Bin}\left(#1\right)}
\newcommand{\BinoMix}[1]{\mathrm{BinMix}\left(#1\right)}
\newcommand{\smallO}[1]{\ensuremath{\mathop{}\mathopen{}{\scriptstyle\mathcal{O}}\mathopen{}\left(#1\right)}}

\newcommand{\reels}{\mathbb{R}} \newcommand{\relatifs}{\mathbb{Z}}
\newcommand{\esp}{\mathbb{E}} \newcommand{\proba}{\mathbb{P}}
\newcommand{\nat}{\mathbb{N}} \newcommand{\pk}{P_K}
 \newcommand{\qk}{Q_K}

\newcommand{\Proba}[1]{\mathbb{P}\left(#1\right)}
\newcommand{\Esp}[1]{\mathbb{E}\left(#1\right)}

\newtheorem{theoreme}{Theorem}[section]
\newtheorem{proposition}[theoreme]{Proposition}
\newtheorem{corollaire}[theoreme]{Corollary}
\newtheorem{assumption}[theoreme]{Assumption}
\newtheorem{example}[theoreme]{Example}
\newtheorem{definition}[theoreme]{Definition} 
\newtheorem{lem}[theoreme]{Lemma}
\newtheorem{remarque}[theoreme]{Remark}
\newtheorem{nota}[theoreme]{Notation}

\usepackage{todonotes}  \newboolean{DisplaySOS}
\setboolean{DisplaySOS}{true} 
\newcommand{\SOS}[1]{\ifthenelse{\boolean{DisplaySOS}}{{\textcolor{blue}{\textbf{Note:} [#1]}}}{}} \makeatletter \def\bbordermatrix#1{\begingroup \m@th \@tempdima 4.75\p@ \setbox\z@\vbox{%
		\def\cr{\crcr\noalign{\kern2\p@\global\let\cr\endline}}%
		\ialign{$##$\hfil\kern2\p@\kern\@tempdima&\thinspace\hfil$##$\hfil
			&&\quad\hfil$##$\hfil\crcr \omit\strut\hfil\crcr\noalign{\kern-\baselineskip}%
			#1\crcr\omit\strut\cr}}%
	\setbox\tw@\vbox{\unvcopy\z@\global\setbox\@ne\lastbox}%
	\setbox\tw@\hbox{\unhbox\@ne\unskip\global\setbox\@ne\lastbox}%
	\setbox\tw@\hbox{$\kern\wd\@ne\kern-\@tempdima\left[\kern-\wd\@ne
		\global\setbox\@ne\vbox{\box\@ne\kern2\p@}%
		\vcenter{\kern-\ht\@ne\unvbox\z@\kern-\baselineskip}\,\right]$}%
	\null\;\vbox{\kern\ht\@ne\box\tw@}\endgroup} \makeatother

\title{	A Mathematical Analysis of Memory Lifetime in a simple Network Model
	of Memory}
\author{Pascal Helson\footnote{pascal.helson@inria.fr}} \date{\today}

\graphicspath{{../images-hal/}}

\begin{document} \maketitle \tableofcontents
	\section*{Abstract} 
	We study the learning of an external signal by a neural network and the time to
	forget it when this network is submitted to noise.
	The presentation of an external stimulus to the recurrent network of binary neurons may change the
	state of the synapses.
	Multiple presentations of a unique signal leads to its learning.
	Then, during the forgetting time, the presentation of other signals (noise) may also modify the synaptic weights.
	We construct an estimator of the initial signal using the synaptic currents
	and define by this way a probability of error.
	In our model, these synaptic currents evolve as Markov chains.
	We study the dynamics of these Markov chains and obtain a lower bound on 
	the number of external stimuli that the network can receive 
	before the initial signal is considered as forgotten
	(probability of error above a given threshold).
	Our results hold for finite size networks as well as in the large size asymptotic.
	Our results are based on a finite time analysis rather than large time asymptotic.
	We finally present numerical illustrations of our results.
	
	\section{Introduction} 
	
	\cite{amit_learning_1994} proposed a model to study the
	memory capacity of neural networks. The main novelty of their work was the online
	learning and forgetting of a sequence of random signals.
	Indeed, in previous models (e.g.~\cite{willshaw_non-holographic_1969}
	or~\cite{hopfield_neural_1982}), signals are stored in a fixed weight matrix. 
	This matrix is determined as a function of signals to learn.
	These models are called associative or attractor neural network (ANN) models: a stimulus is said
	to be stored if its neural representation is an attractor of the neural dynamics.
	The maximum storage capacity of ANN models have been widely studied.
	\cite{gardner1988optimal} computed this capacity for the optimal synaptic weight matrix.
	They showed that maximal storage is obtained for sparse coding.
	Moreover, there has been study of the robustness to noise in the synaptic weight matrix and in
	the initial input. \cite{sommer_bayesian_1998} proposed Bayesian retrieval processes for a
	stochastic version of the Willshaw model.
	However, beyond the maximum number of stimuli learnt, blackout catastrophe (forgetting of all memories)
	appears in ANN models.
	This blackout can be avoided by allowing the plasticity of the synapses.
	
	\indent \cite{amit_learning_1994} proposed the following experimental protocol:
	a neural network, with both binary synapses and binary neurons, receives
	and learns new random stimuli while forgetting the previous ones.
	Every signal may affect the synaptic weights. After a certain amount of time,
	the first stimulus is presented again (priming) and the ability of the network to
	recognize it is questioned: how many stimuli can be presented before it forgets
	the initial signal? To provide an answer, Amit and Fusi
	performed a signal-to-noise ratio
	(SNR) analysis. The signal under consideration is the sum of the synaptic currents onto one neuron
	when the network receives the priming. 
	As~\cite{gardner1988optimal} found in the case of the ANN models,
	Amit and Fusi concluded that the coding of the stimuli needs to be sparse in order to
	optimise the storage capacity.
	They proposed a scaling of the coding level $f$ (probability that a neuron is selective to a signal)
	as a function of the size $N$ of
	the network. According to their retrieval criterion, the optimal coding level is on the order of
	$f \sim \frac{\log(N)}{N}$.
	In the large N asymptotic, what they called the storage capacity is then on the order of $\frac{1}{f^2}$ for depression probabilities proportional to $f$. 
	
	Extensions and approaches different from SNR have then been studied.
	First,~\cite{brunel1998slow} studied a different protocol:
	they fixed the number of random stimuli and presented them randomly multiple times.
	Their analysis relied on the comparison of two quantities: the mean potentiation (MP) and the
	intra-class potentiation (ICP). MP is the mean of synaptic weights.
	ICP is the mean of synaptic weights among synapses involved in the learning of a stimulus.
	Intuitively, when ICP is much larger than MP,
	the trace of a stimulus in the synaptic weights is still non negligible.
	They found two possible loading regimes, a low-loading (resp. high-loading)
	regime with a memory capacity on the order $\frac{1}{f}$ (resp. $\frac{1}{f^2}$).
	\cite{dubreuil_memory_2014} did a deeper analysis of the multiple presentations model of~\cite{brunel1998slow} and the one shot learning model of~\cite{amit_learning_1994}, under the assumption $N$ large and $f$ small.
	Then,~\cite{elliott_memory_2014} considered the mean number of signals presented before
	the synaptic current crosses a fixed threshold: the mean first passage time (MFPT).
	More complex and biologically plausible models have been proposed and analysed numerically
	in the following studies:~\cite{amit_spike-driven_2003,miller_neural_2012,zenke_memory_2014}.
	Finally, to the best of our knowledge, the first article to present a precise way to retrieve
	stimuli is the one of~\cite{amit_precise_2010}.
	In this article, they insisted on the role played by the synaptic correlations
	and proposed a way to compute numerically an approximation of the distributions of the synaptic currents.
	It enables them to introduce a new retrieval criterion based on what they called
	retrieval probabilities.
	
	Inspired by this last article, we study here a statistical test based on the synaptic currents.
	In particular, we study the probability of error associated to this test.
	Such an error has been studied before under some additive assumptions
	on the distribution of synaptic currents.
	\cite{amit_precise_2010} did a numerical analysis with a Gaussian approximation. \cite{dubreuil_memory_2014} gave an analytical result on the probability of no error
	in the large $N$ asymptotic, assuming independence of synapses
	(which leads the synaptic currents to follow Binomial distributions).
	Here, we perform an analytical study of this error without such approximations and we manage to control it
	by extending previous analytical studies of~\cite{amit_learning_1994,amit_precise_2010} on some points.
	First, we give properties of the synaptic current process such as the spectrum of its transition matrix
	(Propositions~\ref{prop-spec-learning} and ~\ref{prop-spec}).
	Moreover, we study the case of multiple presentations of the signal to be learnt.
	Finally, we give in Remark~\ref{rem:main-result_q_fix} and Theorem~\ref{theo:main_result_q_N}
	explicit bounds on the time during which
	a given signal is kept in memory (probability of error below a given threshold).
	These results deal with a broader range of depression probabilities than in the previous
	studies. We summarize our asymptotic results in Remark~\ref{rem:main-result1}.
	
	The rest of the paper is organised as follows.
	We expose the model and the statistical test in Section~\ref{sec-2}.
	After learning one specific signal, the network is submitted to random
	signals responsible for its forgetting.
	The statistical test consists in estimating the initial signal from the pre-synaptic inputs caused
	by priming (using a threshold estimator).
	We measure whether the signal is still in memory by computing the error associated to this test.
	After the formal definition of this error, the main results are presented.
	Then, Section~\ref{sec-disc-analysis} is devoted to deriving a lower bound on the
	maximum number of stimuli one can present while reasonably remembering the
	initial signal.
	This derivation relies on the fact that, asymptotically, as time goes to infinity,
	synaptic currents converge in law to a Binomial mixture (Corollary~\ref{cor:mes-inv-binmix}).
	We assume that, before learning, the synaptic currents follow their stationary distributions.
	Afterwards, the learning phase splits the network in two groups: the neurons activated by the
	signal and the others.
	Then, during the forgetting phase, the laws of the synaptic currents of these two
	groups are shown to remain Binomial mixtures with an explicit dynamics on
	their mixing distributions (Proposition~\ref{prop:dyn-h}).
	In the second part, we evaluate the probability of error of the test and the
	maximum number of stimuli one can present before the test fails
	(Remark~\ref{rem:main-result_q_fix} and Theorem~\ref{theo:main_result_q_N}).
	The computations are based on estimates on the support and on the tail of the mixing distributions.
	Then, we perform numerical simulations in Section~\ref{sec-4}. Finally, technical results
	are proved in the Appendix~\ref{app:proofs}.
	
	\section{The model and the estimator}\label{sec-2} 
	First, we present the neural network model and the protocol followed for learning and forgetting. 
	Then, we define the estimator, derive the equations describing the dynamics of the synaptic
	currents and detail the main assumptions.
	Finally, we present typical numerical simulations at the end of this section.
	
	\subsection{The neural network and the protocol} 
	In order to ease the introduction of the different variables, we suggest the reader to see the model
	as describing an experiment on a person's ability to learn a stimulus.
	In particular, we ask for how long a learnt signal can persist in memory when the person
	is presented some other signals which we termed loosely as noise.
	
	Let us assume that we present a sequence of external stimuli to a network of $N+1$ neurons.
	Thus, we sum over $N$ external synaptic currents to get the total synaptic current onto one neuron.
	We do not study the dynamics of the membrane potential nor the firing rate of neurons,
	but rather we consider their neural activities, $\xi\in \{0,1\}^{N+1}$.
	Hence, the neurons do not have their own dynamics but instead they follow the dynamics
	of the signals.
	We say that the neuron $i$ is selective (resp. not selective) to a signal if its neural response is	
	$\xi^i=1$
	(resp. $\xi^i=0$).
	We assume that a given signal uniquely determines the neural response.
	Therefore, we refer in an equivalent way to stimulus/signal or neural response in the following.
	Signals are assumed to be random and we denote by $(\cdots,\xi_{-1},\xi_{0}, \xi_{1}\cdots)$
	the corresponding sequence.
	We call $t$ the time at which the $t^{th}$ signal after $\xi_0$ is shown.
	We assume that the $\xi_t$s are independent and identically distributed
	(\textit{i.i.d.}) random variables (\textit{r.v.}) in \(\{0,1\}^{N+1}\).
	Moreover, for each \(t\), the components \(\xi^1_t, \cdots, \xi^{N+1}_t\) of \(\xi_t\) are
	themselves \textit{i.i.d.} with Bernoulli distribution with parameter \(f\):
	
	\(
	\qquad \forall i\in \llbracket 1,N+1\rrbracket, \qquad \proba\left(\xi_t^i=1\right)= f
	= 1- \proba\left(\xi_t^i=0\right).
	\)
	
	The synaptic weight from neuron $j$ to neuron $i$ at time $t$ is denoted by $J_t^{ij}$.
	It can only take two values $J_-<J_+$ and we denote by $J_{t}=\{J_t^{ij},i\neq j\}\in
	\{J_-,J_+\}^{N(N+1)}$ the matrix of synaptic weights.
	We consider a plasticity rule which can be viewed as a classic Hebbian rule.
	The law of $J_{t+1}$ only depends on $J_{t}$ and $\xi_t$.
	The corresponding transition probabilities are 
	\begin{itemize} 
		\item
		$\mathbb{P}\left(J_{t+1}^{ij}=J_+|J_t^{ij}=J_-,(\xi_t^i,\xi_t^j)=(1,1)\right) =
		q^+$,
		\item
		$\mathbb{P}\left(J_{t+1}^{ij}=J_-|J_t^{ij}=J_+,(\xi_t^i,\xi_t^j)=(0,1)\right)
		= q_{01}^-$, 
		\item
		$\mathbb{P}\left(J_{t+1}^{ij}=J_-|J_t^{ij}=J_+,(\xi_t^i,\xi_t^j)=(1,0)\right)
		= q_{10}^-$. 
	\end{itemize}
	The transition probabilities not mentioned here and involving the change
	of state of a synaptic weight are set to zero.
	For example, $\mathbb{P}\left(J_{t+1}^{ij}=J_-|J_t^{ij}=J_+,(\xi_t^i,\xi_t^j)=(0,0)\right) = 0$.
	In order to simplify the notation and without loss of generality, we set:
	
	\( 
	\qquad \qquad \qquad J_- = 0  \text{ (weak synapse) and } J_+ = 1 \text{ (strong synapse)}.
	\)
	
	Moreover, in order to avoid critical cases, we also assume that
	\begin{equation}\label{ass-param}
	f,q^-_{01},q^-_{10},q^+\in (0,1]. 
	\end{equation}
	The parameters $q^-_{01}$ and $q^-_{10}$ represent respectively the homosynaptic and heterosynaptic
	depressions, see~\cite{brunel1998slow}. 
	
	We now give the protocol to learn and then forget a signal.
	We denote by $\xi_0$ the signal to be learnt.
	Before presenting it, we assume that the network has received a lot of random signals thereby driving
	the law of the synaptic weights matrix in its ``stable'' state at time $t=-r+1$
	(we prove in Proposition~\ref{prop:mes-inv-total-proc} that there is a unique invariant measure).
	The learning phase consists in performing $r$ presentations of $\xi_0$.
	In order to be consistent with the previous  description, the sequence of presented stimuli is then
	$(\cdots,\xi_{-r},\underbrace{\xi_{0},\cdots,\xi_{0}}_{r\ \text{ times}},  \xi_{1}, \xi_{2}\cdots)$
	that is \(\xi_{t} =  \xi_{0}\) for \(t\in \llbracket-r+1,0\rrbracket\).
	The presentation of the subsequent signals leads to the forgetting of $\xi_0$.
	
	\subsection{Presentation of the estimator}\label{subsec-pres-estimator}
	We study the consistency through time of the response of one neuron to the initial signal. 
	To do so, we consider the previous protocol.
	After the repetitive presentation of $\xi_0$, the signal has left a certain footprint
	in the matrix $J_1$.
	This trace is subsequently erased by the presentation of the following signals.
	How much information from a stimulus learnt is left at time $t$?
	As an answer, we define a probability of error.
	This error is associated to a decision rule based on the projection of $J_t$ on $\xi_0$.
	For the neuron $i$, such a projection at time \(t\) is given by $\sum_{j\neq i}J_{t}^{ij}\xi_0^j$.
	In this framework, neurons are similar.
	Hence, in order to simplify the notation and without loss of generality,
	our study focuses on neuron $i=1$.
	We denote by $h_t$,
	\begin{equation}\label{eq:def-h_t}
	h_t=\sum_{j=2}^{N+1}J_{t}^{1j}\xi_0^j,
	\end{equation}
	the synaptic current onto neuron $1$ when presenting again $\xi_0$ at time $t$.
	In this framework, the initial signal is presented in a fictive way.
	This means that the synaptic weights do not change following this fictitious presentations. 
	Note that the process $\left(h_t\right)_{t\geq 0}$ strongly depends on the initial
	number $K$ of active neurons 
	\begin{align}\label{def:K}
	K = \sum_{j=2}^{N+1} \xi_0^j.
	\end{align}
	We denote by $h_{t,K}$ the process equal in law to $h_t$ knowing $K$:
	$h_{t,K}\overset{\mathcal{L}}{=}\left(h_t|K\right)$.
	The process $h_{t,K}$ is Markovian, see Proposition~\ref{prop-Markov}.
	
	We define a threshold estimator
	$\hat{\xi}: \nat^* \times \llbracket  0,N \rrbracket  \rightarrow \{0,1\}$ such that
	$\hat{\xi}(t,\theta) = \mathbbm{1}_{h_{t}>\theta}$ with associated probability of errors:
	\begin{align*}
	p_e^0(t,\theta) &= \proba\left(\hat{\xi}(t,\theta) = 1\ | \ \xi_0^1=0\right)= \proba\left(h_t > \theta\ | \ \xi_0^1=0\right), \\
	p_e^1(t,\theta) &= \proba\left(\hat{\xi}(t,\theta) = 0\ |
	\ \xi_0^1=1\right) = \proba\left(h_t \leq \theta\ | \ \xi_0^1=1\right).
	\end{align*}
	\begin{nota}\label{not:h-tK-y}
		We denote by
		\(h_t^y \overset{\mathcal{L}}{=}\left(h_t | \
		\xi_0^1=y\right)
		\quad \text{and} \quad
		h_{t,K}^y\overset{\mathcal{L}}{=}\left(h_t|\xi_0^1=y,K\right)
		.\)
	\end{nota}
	\noindent In the following, we shall use the plural "distributions of $h_{t,K}^y$"
	to say distributions of $h_{t,K}^0$ and $h_{t,K}^1$.
	The probability of error $p_e^0(t,\theta)=\proba\left(h_t^0>\theta\right)$
	(resp. $p_e^1(t,\theta)=\proba\left(h_t^1\leq\theta\right)$) corresponds to the probability that
	the estimator responds positively (resp. negatively) to the priming presented at time $t>0$
	whereas the neuron was not activated (resp. activated) initially.  
	We aim at evaluating these errors: for fixed $\delta\in (0,1)$, we estimate the largest time \(t_*\)
	such that both $p_e^0$ and $p_e^1$ are smaller than \(\delta\) up to time \(t_*\), 
	\begin{equation}\label{eq:deftetoile}
	t_*(\delta,r,N) := \max_{\theta \in \llbracket  0,N \rrbracket} 
	\left(\inf\left\{t\geq 1,\, p_e^0(t,\theta) \vee p_e^1(t,\theta) \geq \delta\right\}\right),
	\end{equation}
	where $x \vee y = \max(x,y)$ and $x \wedge y = \min(x,y)$.
	
	\subsubsection*{Main Results (informal)}
	
	\textit{
		For any fixed error $\delta \in (0,1)$, there is an unbounded set of couples
		$(N,r) \in {\nat^*}^2$ for which we show the existence of a threshold
		$\theta_{\delta,r,N} \in \{ 0, 1, \hdots, N \}$ ensuring 
		\[
		t_{*}(\delta,r,N) \geq \inf\left\{t\geq 1,\, p_e^0(t,\theta_{\delta,r,N}) \vee
		p_e^1(t,\theta_{\delta,r,N}) \geq \delta\right\} \geq \hat{t}(\delta,r,N),
		\]
		where an explicit formula of $\hat{t}$ is given in Remark~\ref{rem:main-result_q_fix} for
		fixed potentiation and depression probabilities.
		Another formula of $\hat{t}$ is given in Theorem~\ref{theo:main_result_q_N}
		for depression probabilities depending on $N$.
		In particular, assuming that the depression probabilities are proportional to the coding level $f$, we obtain that $\hat{t}(\delta,r,N)$ is on the order of $\frac{1}{f^2}$.
	}
	
	\noindent The proofs of these results rely on the study of the Markov chains
	$\left(h_{t,K}\right)_{t\geq 1}$ and $\left(h_{t,K}^y\right)_{t\geq 1}$.
	\begin{proposition}\label{prop-Markov}
		The chains $\left(h_{t,K}\right)_{t\geq 1}$ and $\left(h_{t,K}^y\right)_{t\geq 1}$ are Markovian.
		At the end of the learning phase, we have
		\begin{align}\label{def-h-t0-Bin} 
		\begin{split}
		h_{1,K} \overset{\mathcal{L}}{=} h_{-r+1,K} &+ \xi^1_{0}\Bin\left(K-h_{-r+1,K},
		1-(1-q^+)^r\right)
		\\
		&-(1-\xi^1_{0})\Bin\left(h_{-r+1,K}, 1-(1-q_{01}^-)^r\right)
		\end{split}
		\end{align}
		where, conditionally on $h_{-r+1,K}$, the two Binomial random variables are independent.
		And during the forgetting phase, for all $t\geq 1$:
		\begin{multline}\label{def-h-Bin}
		h_{t+1,K}
		\overset{\mathcal{L}}{=} h_{t,K} + \xi^1_{t}\left[\Bin\left(K-h_{t,K},f
		q^+\right)-\Bin\left(h_{t,K},(1 - f) q_{10}^-\right)\right]\\
		-(1-\xi^1_{t})\Bin\left(h_{t,K},f q_{01}^-\right) 
		\end{multline}
		where, conditionally on $h_{t,K}$, the three Binomial random variables are independent.
		
		The Markov chains $\left(h_{t,K}^y\right)_{t\geq 1}$ for $y\in\{0,1\}$, satisfy the
		equation~\eqref{def-h-t0-Bin} with $\xi_0^1 = y$ and the
		equation~\eqref{def-h-Bin}.
	\end{proposition} 
	\begin{proof}
		In order to study the jump from $h_{t,K}$ to $h_{t+1,K}$, we count the synapses
		that potentiate and the ones that depress upon presenting a signal $\xi_t$.
		From the definitions~\eqref{eq:def-h_t} and~\eqref{def:K} of $h_t$ and $h_{t,K}$,
		we only need to consider the $K$ synapses $J_t^{1j}$ with $j\geq 2$ such that $\xi_0^j=1$.
		At time $t$, there are $h_{t,K}$ strong synapses and $K-h_{t,K}$ weak synapses.
		Given $\xi_t^1$ and $h_{t,K}$, every synapse evolves independently following a Bernoulli law. 
		
		From time $-r+1$ to $1$, if $\xi_0^1=0$, every strong synapse is $r$ times candidate to
		depression so it has probability \(1 - (1-q_{01}^-)^r\) to depress.
		If $\xi_0^1=1$, every weak synapse is $r$ times candidate to potentiation so it has probability 
		$1 - (1-q^+)^r$ to potentiate.
		Equation~\eqref{def-h-t0-Bin} follows.
		
		From time $t\geq 1$ to $t+1$, if $\xi_t^1=0$,
		the probability that a strong synapse depresses is $f q_{01}^-$.
		If $\xi_t^1=1$, the probability that a weak synapse potentiates is $fq^+$ and
		the probability that a strong synapse depresses is $(1-f) q_{10}^-$.
		Equation~\eqref{def-h-Bin} follows.
		
		By definition of $h_{t,K}^y$, the chain satisfies the equations~\eqref{def-h-t0-Bin}
		with $\xi_0^1 = y$ and \eqref{def-h-Bin}.
	\end{proof}
	
	\begin{corollaire}\label{prop-conv-mes-inv}
		Assume that~\eqref{ass-param} holds.
		Then, for all $K \in \llbracket0,N\rrbracket$, the Markov chain \((h_{t,K})_{t\geq 1}\)
		admits a unique invariant measure \(\pi_K\) with support in $\llbracket0,K\rrbracket$.
		Moreover, for any initial condition \(h_{0,K}\), the Markov chain \((h_{t,K})_{t \geq 1}\)
		converges in law to \(\pi_K\).
		In addition, the chain $\left(h_t\right)_{t\geq 1}$ converges in law to
		$\pi_{\infty} = \sum_{K=0}^{N} \Proba{\hat{K} = K} \pi_{K}$ where
		$\hat{K} = \sum_{j=2}^{N+1} \xi_0^j$.
	\end{corollaire} 
	\begin{proof} 
		By~\eqref{ass-param}, the Markov chain $\left(h_{t,K}\right)_{t\geq 1}$ is irreducible and 
		aperiodic on a finite state space.
		Thus, it admits a unique invariant measure towards which it converges.
		
		Let $\hat{K} = \sum_{j=2}^{N+1} \xi_0^j$.
		From the Bayes' formula we get that for all $l\in \llbracket 0,N \rrbracket$,
		\begin{align*}
		\lim\limits_{t \rightarrow \infty}\Proba{h_t=l} &= \lim\limits_{t \rightarrow \infty}
		\left(\sum_{K=0}^N\Proba{\hat{K} = K} \Proba{h_{t,K}=l}\right)
		= \sum_{K=0}^{N} \Proba{\hat{K} = K} \pi_{K}(l).
		\end{align*}
	\end{proof}
	
	\begin{remarque}\label{rem:20190821} 
		The Markov chains $\left(h^y_{t,K}\right)_{t\geq 1}$ have the same transition matrix as
		$\left(h_{t,K}\right)_{t\geq 1}$.
		They differ by their distribution at time $t=1$.
		Hence, they all converge to $\pi_K$.
		Moreover, both $\left(h_t^0\right)_{t\geq 1}$ and
		$\left(h_t^1\right)_{t\geq 1}$ converge in law to $\pi_{\infty}$.
	\end{remarque}
	
	\begin{proposition}\label{prop:mes-inv-total-proc}
		Under the assumption~\eqref{ass-param}, the process $\left(\xi_{t},J_t\right)_{t\geq 1}$
		converges to its unique invariant measure.
		We denote it by $\rho_{\infty}$.
	\end{proposition}
	\begin{proof}
		Same argument as for Corollary~\ref{prop-conv-mes-inv}.
	\end{proof}
	We now give the main assumptions.
	\begin{assumption}\label{ass-main} \hspace{2em}
		\begin{enumerate}[label=\ref{ass-main}.\arabic*] \item
			$
			\left(\xi_{0},J_{-r+1}\right) \overset{\mathcal{L}}{=}   \rho_{\infty} \quad
			\text{ and in particular } \quad h_{-r+1,K}, h_{-r+1,K}^0,h_{-r+1,K}^1
			\overset{\mathcal{L}}{=}   \pi_K.
			$
			\label{ass-main:init-cond}
			\item Assume that $f$ depends on $N$.
			Let us denote it by $f_N$ such that \(\lim\limits_{N\infty} f_N = 0\)
			and \(\lim\limits_{N\infty} Nf_N = +\infty\). \label{ass-main:on-f}
			\item Let $q_{01,N}^- = a_Nf_N \text{ and } q_{10,N}^- = b_N f_N$
			with $a_N,b_N: \nat^*\rightarrow \reels$ such that $a_N,\ b_N$ both converge in $[0,+\infty)$.
			However, we assume that at least one of the two limits is not $0$
			and\label{ass-main:on-q10-q01}
			\[
			\lim\limits_{N \infty}
			\ q_{01,N}^- = \lim\limits_{N \infty}\ q_{10,N}^-
			= \lim\limits_{N \infty} \frac{b_N^2}{N f_N a_N}
			= \lim\limits_{N \infty} \frac{b_N}{N f_N}
			= 0, \quad
			\lim\limits_{N \infty} N f_N a_N = +\infty.
			\]
		\end{enumerate} 
	\end{assumption} 
	We consider a general paradigm in which before receiving the stimulus $\xi_0$,
	many stimuli have already been sent \((\cdots, \xi_{-r-2}, \xi_{-r-1}, \cdots)\). 
	We assume that the process $\left(\xi_{t},J_t\right)_{t\leq -r+1}$ has reached
	its invariant measure at time \(t=-r+1\) by Assumption~\ref{ass-main:init-cond}. 
	Then, one key parameter is the coding level $f$.
	We assume that it depends on $N$ in the analysis of the large $N$ asymptotic:
	Assumption~\ref{ass-main:on-f}.
	This assumption refers to sparse coding as $f_N$ tends to $0$.
	An additional constraint put forward is that the mean number of selective neurons,
	$Nf_N$, needs to be large enough: Assumption~\ref{ass-main:on-f}.
	In this context, we are interested to see how the dependence on $N$ of the depression
	probabilities can affect the memory lifetime, see Assumption~\ref{ass-main:on-q10-q01}.
	This assumption gives conditions on the large $N$ asymptotic behaviours of the depression
	probabilities.
	
	\subsubsection*{First illustrations}\label{subsec-first-int} 
	In this subsection, we illustrate the dynamics of $\left(h^y_{t,K}\right)_{t\geq 0}$
	and $\left(h_{t,K}\right)_{t\geq 0}$.
	In particular, we are interested in the effects of the coding level $f_N$ on these synaptic currents.
	Let us assume that the signal $\xi_0$ is of size $K=\lfloor f_NN \rfloor$,
	where the floor function $\lfloor x \rfloor$ is equal to $k\in \relatifs$ if $k\leq x < k+1$.
	Let us have a look at the expected size of jumps of $h_{t,K}$
	from the formulas~\eqref{def-h-t0-Bin}, \eqref{def-h-Bin}.
	\begin{align*}
	\text{For } t=1,\quad &\esp\left[h_{1,K}-h_{-r+1,K}|h_{-r+1,K}, \xi_0^1=0\right] 
	= -h_{-r+1,K}(1 - (1- q_{01}^-)^r),
	\\
	&\esp\left[h_{1,K}-h_{-r+1,K}|h_{-r+1,K}, \xi_0^1=1\right] 
	= (K-h_{-r+1,K})(1 - (1- q^+)^r),
	\\
	\forall t>1,\quad &\esp\left[h_{t+1,K}-h_{t,K}|h_{t,K}\right]
	= (K - h_{t,K})f_N^2q^+ -h_{t,K}f_N(1-f_N)(q_{10}^- + q_{01}^- ).
	\end{align*}
	From these equations, we note that the average jump size strongly depends on $f_N$.
	When $f_N$ is close to $1$, the reception of $\xi_0$ has a large impact on the weight matrix,
	easy to	detect.
	However, the following average jump size are close to the initial one.
	Thus, as soon as some other stimuli are presented, the initial signal is forgotten:
	the distributions of $h_{t,K}^0$ and $h_{t,K}^1$ quickly overlap.
	Conversely, when $f_N$ is close to $0$,
	the average jump size is significantly different between the learning (relatively big jumps)
	and the forgetting (relatively small jumps) phases.
	As a consequence, the convergence to the stationary distribution, and thus forgetting, is slower.
	However, the learning still occurs: the initial jump is still big.
	In order to illustrate these phenomena, we plot simulation results obtained with a high coding level,
	$f_N = 0.8$ in Figure~\ref{image.f-big}, and a low coding level,
	$f_N = 0.1$ in Figure~\ref{image.f-small}.
	
	Figure~\ref{sub:trajectory-f08} shows that the size of jumps is effectively big for $f_N = 0.8$,
	just after learning as well as during forgetting time.
	Figure~\ref{sub:init-f08} illustrates the separation between the initial distributions of
	$h_{t,K}^0$ and $h_{t,K}^1$.
	Indeed, at time $t=-r+1=0$, both $h^0_{0,K}$ and $h^1_{0,K}$ follow
	the invariant measure plotted in black.
	Then, after the reception of $\xi_0$, the distribution of $h^0_{1,K}$ is
	shifted to the left and the distribution of $h^1_{1,K}$ to the right.
	Initially, the signal is learnt because the distributions are well separated,
	see Figure~\ref{sub:init-f08}.
	Figures~\ref{sub:time2-f08} and \ref{sub:time4-f08} exhibit the fast overlapping
	of these two distributions.
	Indeed, following the learning phase, the reception of new stimuli makes
	the two distributions converge back quickly to the invariant distribution.
	At time $t=5$, the signal is already forgotten.
	\begin{figure}[h]
		\centering
		\subfloat[]{
			\includegraphics[width=0.45\textwidth]{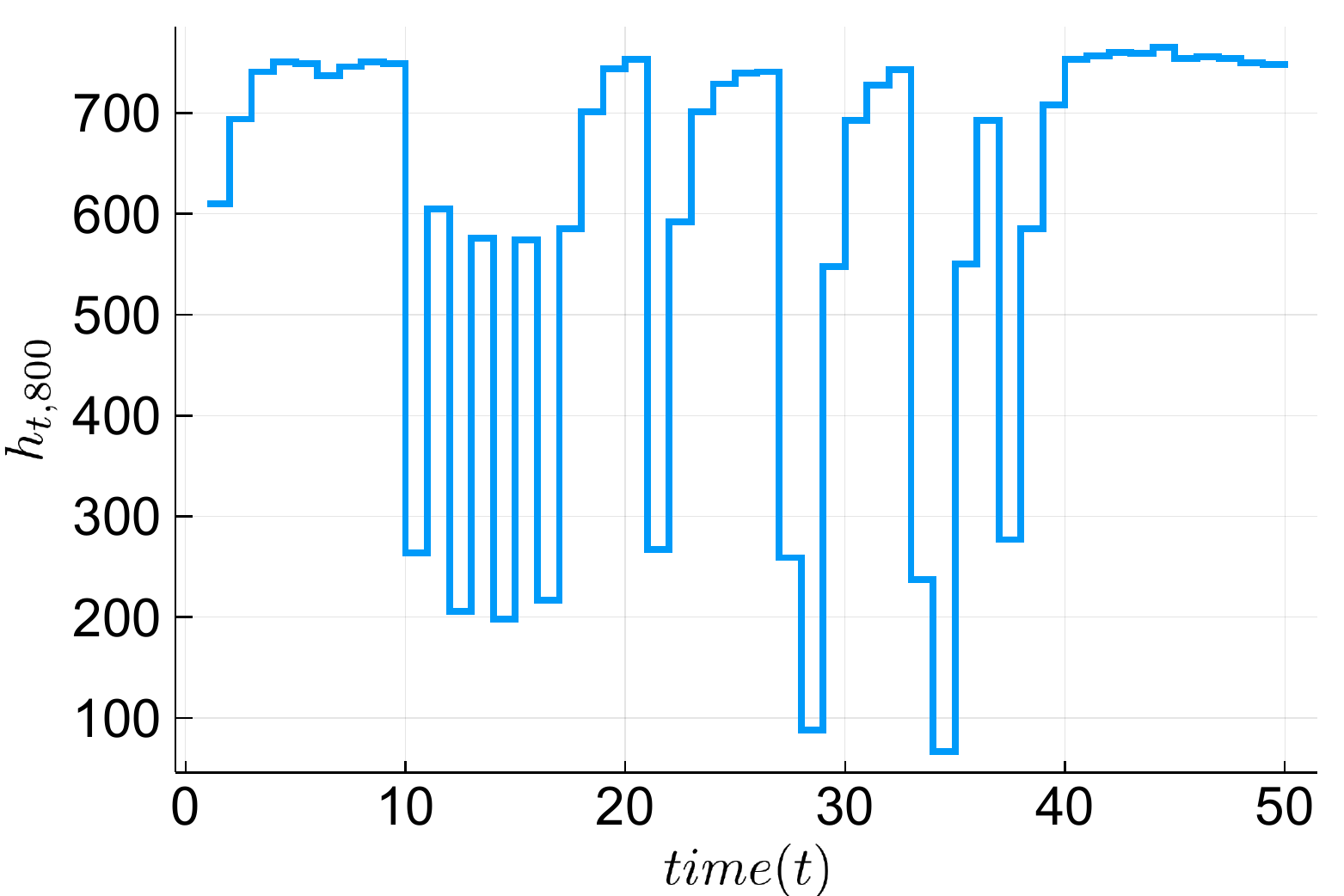} \label{sub:trajectory-f08} }
		\subfloat[]{ \includegraphics[width=0.45\textwidth]{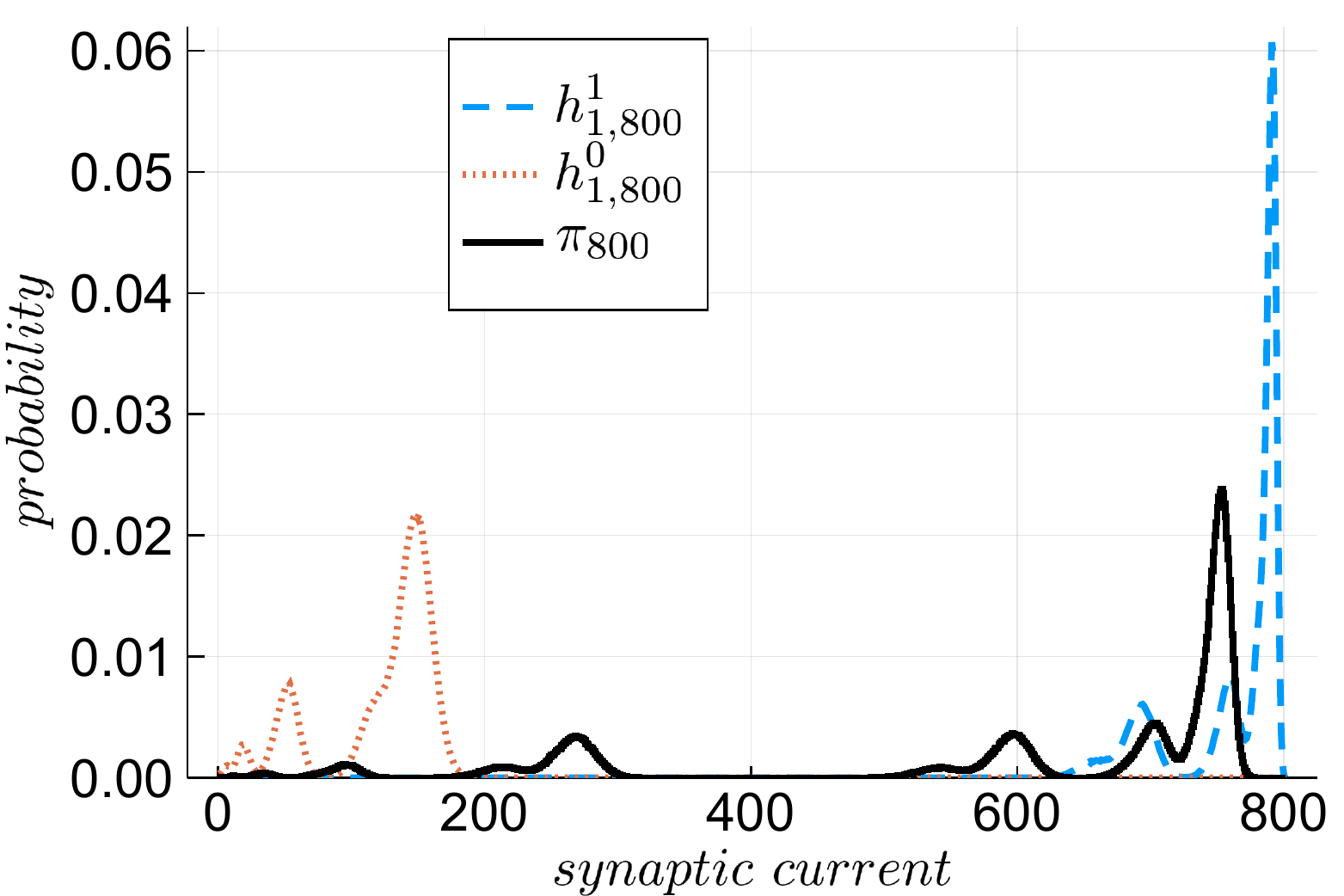} \label{sub:init-f08} }
		\\ 
		\subfloat[]{ \includegraphics[width=0.45\textwidth]{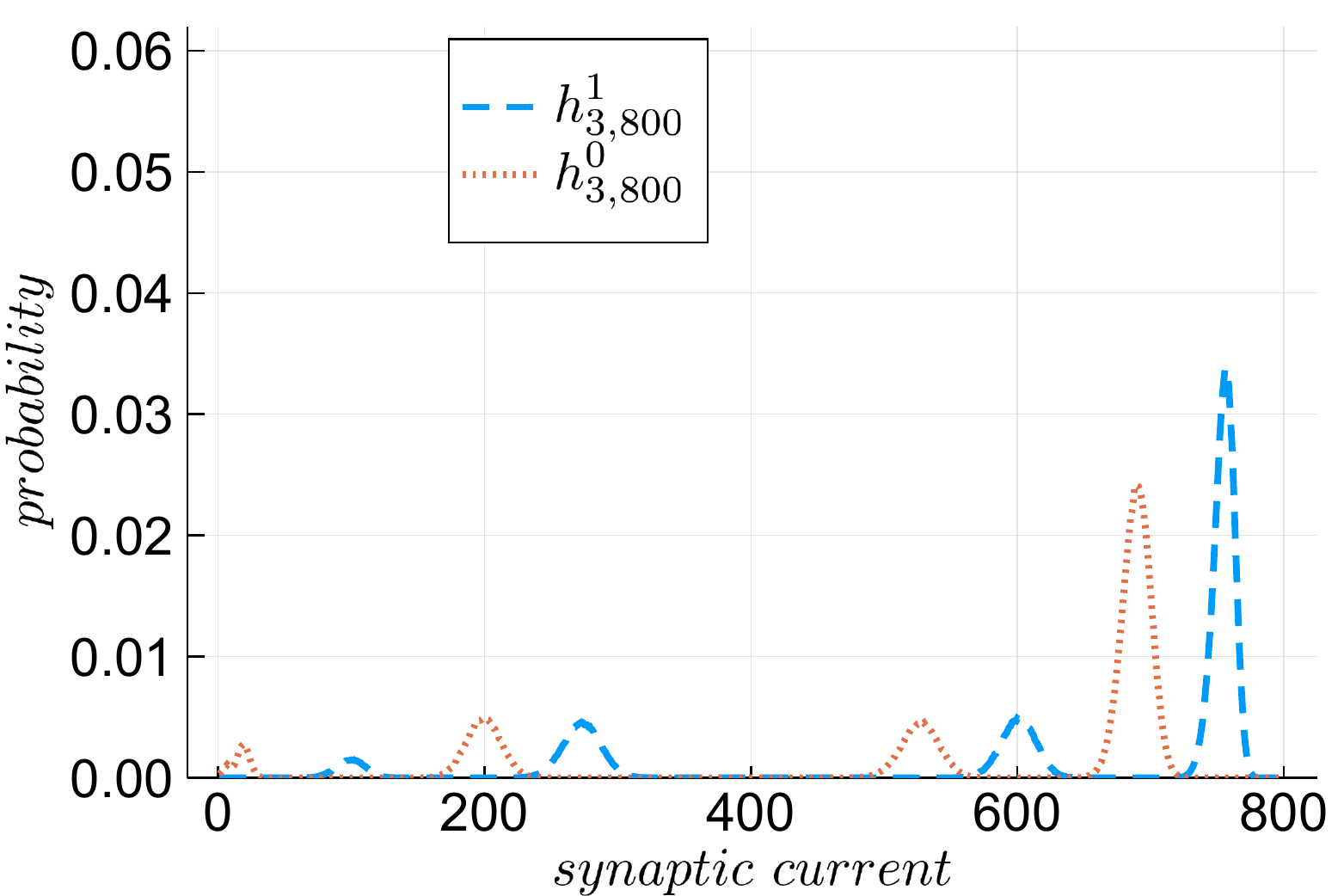}
			\label{sub:time2-f08} } \subfloat[]{
			\includegraphics[width=0.45\textwidth]{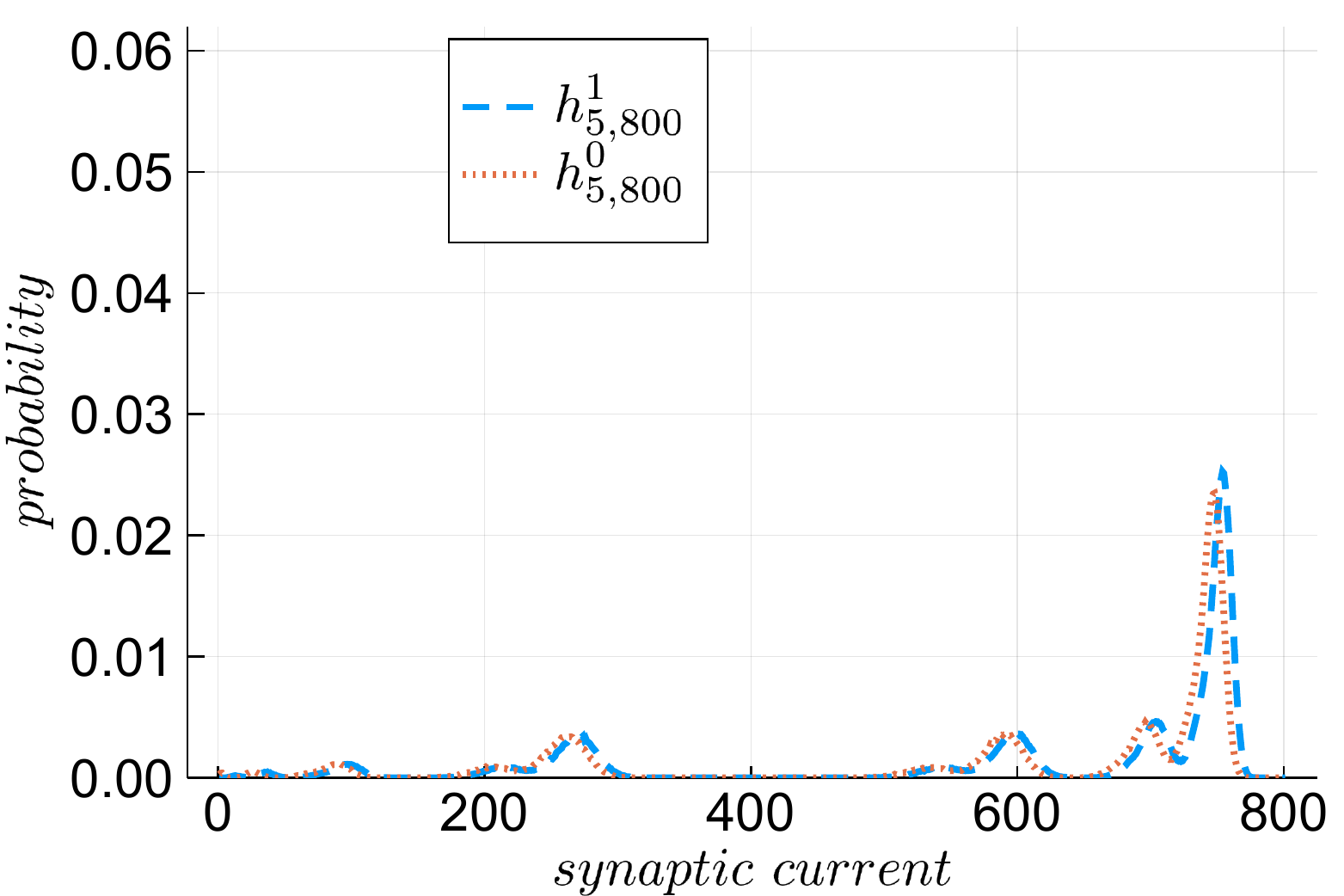} \label{sub:time4-f08} }
		\caption{
			\eqref{sub:trajectory-f08}
			A typical trajectory of $h_{t,800}$.
			\eqref{sub:init-f08}
			The distributions of $h^y_{1,800}$ and the invariant measure $\pi_{800}$.
			\eqref{sub:time2-f08},\eqref{sub:time4-f08}
			The distributions of $h^y_{t,800}$ at time $t=3$ and $t=5$.
			Parameters: $r=1$, $N = 1000$, $K=800$, $f_N=0.8$, 
			$q^+ = 0.8,\ q^-_{01} = 0.8$ and $q^-_{10} = 0.2$.
		}\label{image.f-big} 
	\end{figure}
	Figure~\ref{image.f-small} illustrates the advantages of a low coding level.
	Indeed, even at time $t=20$, the two distributions do not overlap a lot
	and they remained uni-modal.
	This makes the choice of a threshold estimator reasonable.
	Moreover, such an estimator allows a tractable analysis.
	\newpage
	\begin{figure}[h] 
		\centering
		\subfloat[]{ \includegraphics[width=0.45\textwidth]{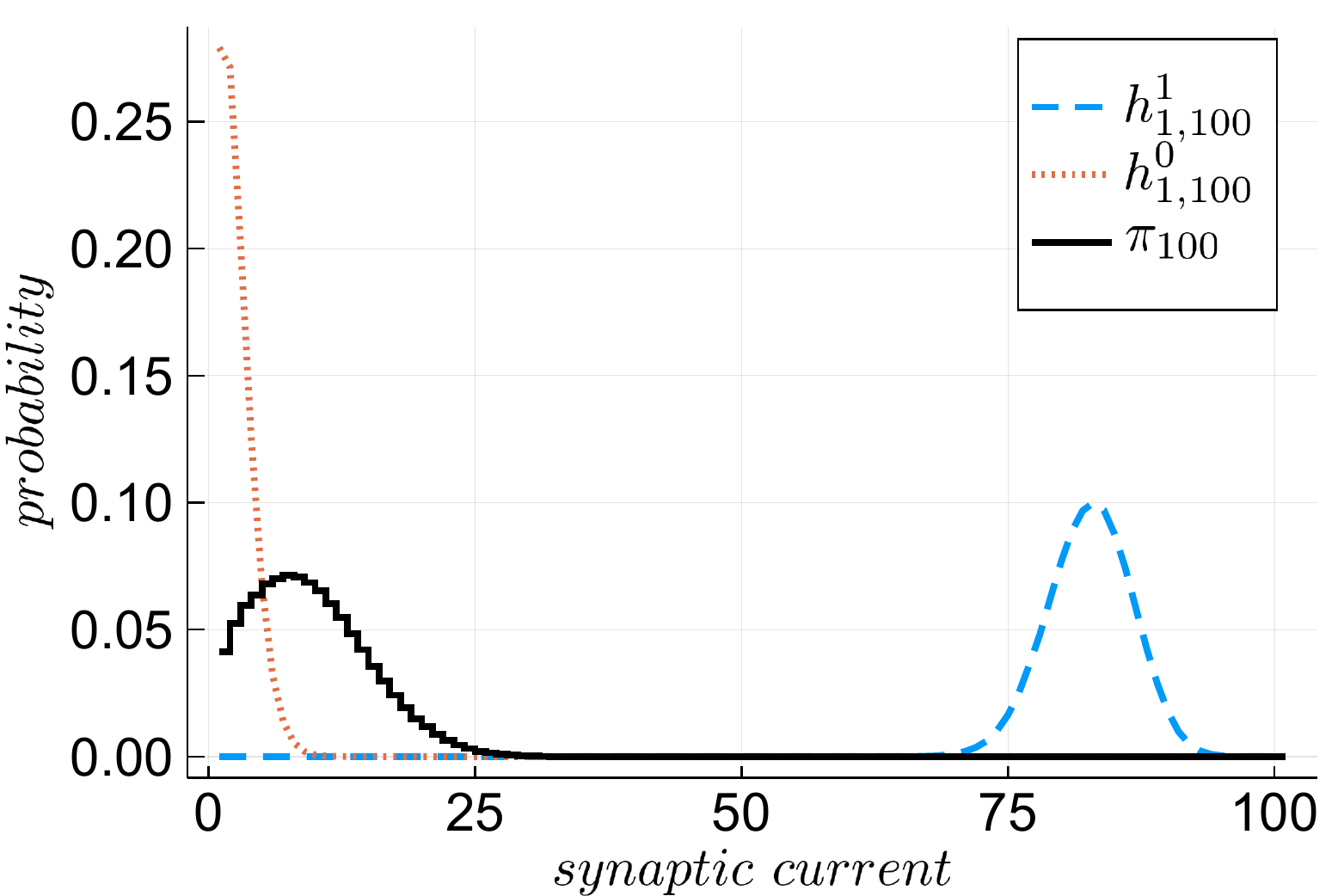}
			\label{sub:init-f01-modif-ink} } \subfloat[]{
			\includegraphics[width=0.45\textwidth]{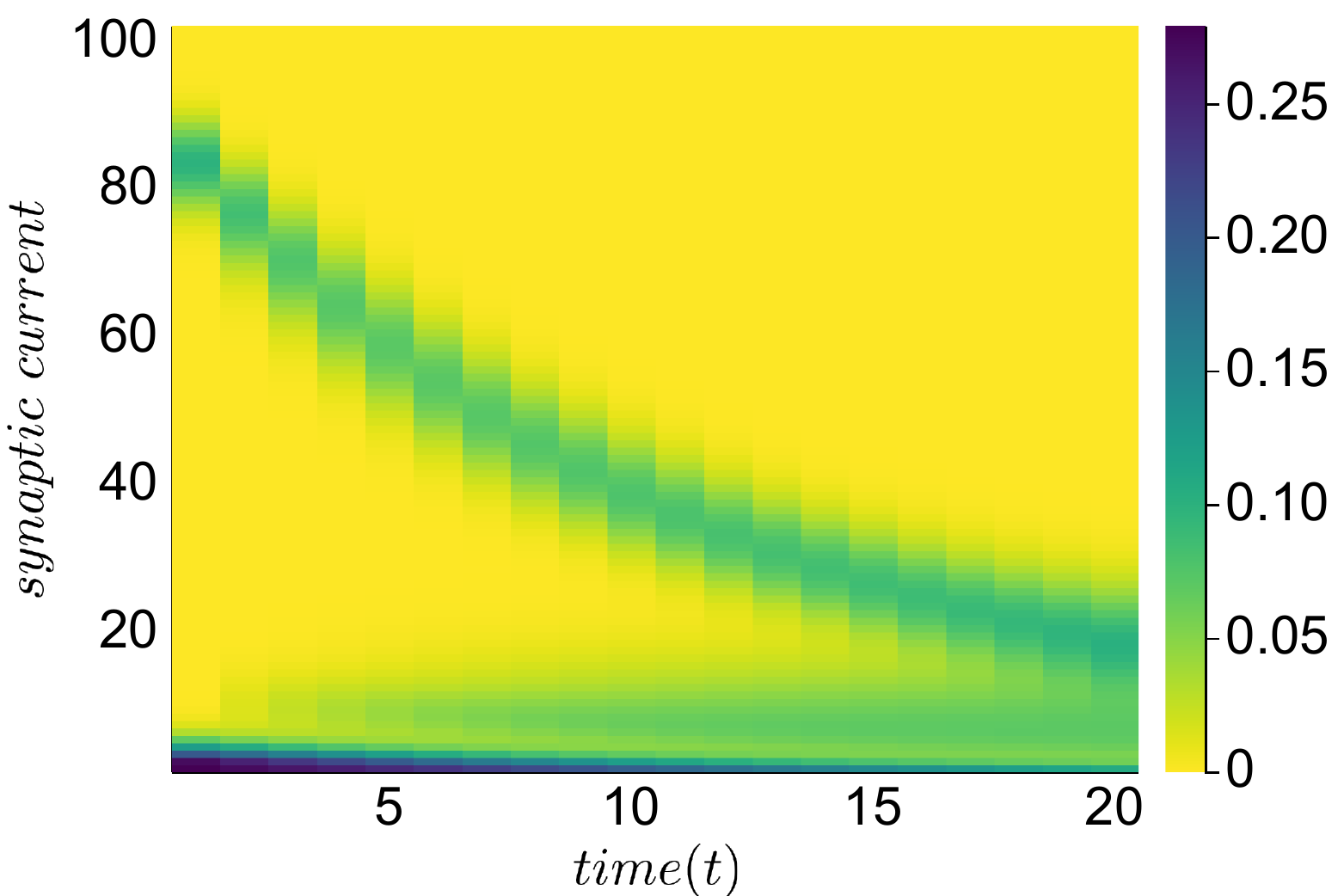} \label{sub:time20-f01} }
		\caption{
			\eqref{sub:init-f01-modif-ink} 
			The distributions of $h^y_{t,K}$ at time $t=1$ and the invariant measure $\pi_{100}$.
			\eqref{sub:time20-f01}
			The sum of the two distributions $h^0_{t,K}$ and $h^1_{t,K}$ for $t\in [1,20]$.
			The colour bar gives the probability values.
			Parameters: $r=1$, $N = 1000$, $K=100$, $f_N=0.1,\ q^+ =
			0.8,\ q^-_{01} = 0.8,$ and $ q^-_{10} = 0.2$.
		}\label{image.f-small} 
	\end{figure}
	
	\section{Results}\label{sec-disc-analysis}
	In this section, we first give some properties satisfied by the distributions
	of $h^y_{t,K}$, see Notation~\ref{not:h-tK-y}, and the invariant measure $\pi_K$.
	They enable us to prove Theorem~\ref{theo:main-result}, and then our main results, in the second part.
	\subsection{Binomial mixture}
	We denote by $F_{[0,1]}$ the set of cumulative distribution functions
	associated to  the set $\mathcal{P}([0,1])$ of probability measures on $[0,1]$. 
	\begin{definition}\label{def-bin-mix}
		The distribution of $X$ is said to be a Binomial mixture with mixing distribution
		$g\in \mathcal P([0,1])$ and size parameter $K$, denoted by $\BinMix(K,g)$, if 
		\begin{align*}	
		\forall j\in \llbracket 0,K\rrbracket,
		\qquad
		\proba\left(X=j\right) = \displaystyle\binom{K}{j} \int_{0}^{1}u^j(1-u)^{K-j}g(du).
		\end{align*}
	\end{definition}
	\begin{remarque}\label{rem-equivalence-BinMix-Bin} \hspace{2em}
		\begin{itemize} 
			\item
			$X\overset{\mathcal{L}}{=} \BinMix(K,g)$ is equivalent to $X|Y\overset{\mathcal{L}}{=} \Bin(K,Y)$
			where $Y$ is a random variable independent of the Binomial and with law $g$.
			Indeed 
			\begin{align*}	
			\proba\left(X=j\right) =
			\int_{0}^{1}\proba(X=j|Y=u)g(du)=\displaystyle \binom{K}{j}
			\int_{0}^{1}u^j(1-u)^{K-j}g(du).
			\end{align*}
			We use both notations $X\overset{\mathcal{L}}{=} \BinMix(K,g)$
			and $X\overset{\mathcal{L}}{=} \BinMix(K,Y)$.
			\item The law of \(X\) is fully characterized by the moments
			\(\mathbb{E}(Y), \mathbb{E}(Y^2),\cdots, \mathbb{E}(Y^K)\).
			Hence, if $\tilde{g} \in \mathcal P([0,1])$ is such that
			$\forall k\in \llbracket0,K\rrbracket,\ \ \int_0^1 u^k\tilde{g}(du)=\int_0^1 u^kg(du),$
			then $\BinMix(K,\tilde{g})\overset{\mathcal{L}}{=}\BinMix(K,g)$.
		\end{itemize} 
	\end{remarque}
	First, we show that the set of Binomial mixtures is stable by the Markov chain \(h_{t,K}\):
	assume that $h_{t,K}\overset{\mathcal{L}}{=} \BinMix(K,g_t)$ for some \(g_t\in \mathcal P([0,1])\),
	then there exists a probability \(g_{t+1}\), function of \(g_t\),
	such that $h_{t+1,K}\overset{\mathcal{L}}{=} \BinMix(K,g_{t+1})$.
	Moreover, denoting by \(G_t\) the cumulative distribution function associated to \(g_t\),
	we show that for all $t\geq 1$, $G_{t+1}(x) = \mathcal{R}(G_t)(x)$ where 
	\begin{nota}\label{func_eq}
		For all $\Gamma\in F_{[0,1]}$ and $u\in \reels,$ $\mathcal{R}$ is defined by
		\begin{equation*} 
		\mathcal{R}(\Gamma)(u) \overset{def}{=}
		f_N \Gamma\left(\frac{u-f_N q^+}{1-(1 - f_N) q_{10}^- - f_N q^+}\right)
		+ (1-f_N) \Gamma\left(\frac{u}{1-f_Nq_{01}^-}\right). 
		\end{equation*}
	\end{nota} 
	
	\begin{proposition}\label{prop:dyn-h} 
		Let us assume that $h_{-r+1,K}\overset{\mathcal{L}}{=} \BinMix(K,g_{-r+1})$,
		for $g_{-r+1}\in \mathcal P([0,1])$.
		Then for all $t\geq 1$, $\exists g_t,\ g_t^0,\ g_t^1 \in \mathcal P([0,1])$ such that
		$h_{t,K}\overset{\mathcal{L}}{=} \BinMix(K,g_t)$ and
		$h_{t,K}^y\overset{\mathcal{L}}{=} \BinMix(K,g_t^y)$ for $y=0,1$.
		Moreover, at time $t=1$,
		\begin{align}
		G_{1}(u) &= f_N G_{-r+1}\left(\frac{u-1}{(1-q^+)^r}+1\right) +
		(1-f_N)G_{-r+1}\left(\frac{u}{(1-q^-_{01})^r}\right),\label{eq:recursive-h-t0}\\
		G_{1}^1(u) &= G_{-r+1}\left(\frac{u-1}{(1-q^+)^r}+1\right)\quad
		\text{and}\quad G_{1}^0(u) =
		G_{-r+1}\left(\frac{u}{(1-q^-_{01})^r}\right),\label{eq-cdf-H-2}\\
		\text{and }\ \forall t\geq &1, \quad
		G_{t+1}(u) = \mathcal{R}(G_t)(u) \quad \text{and} \quad G_{t+1}^y(u) = \mathcal{R}(G_t^y)(u).\label{eq-cdf-H-1}
		\end{align}
	\end{proposition}
	\begin{remarque}
		First, we note that $g_t$ does not depend on $K$.
		This is crucial for the proof of Theorems~\ref{theo:main-result}
		and~\ref{theo:main_result_q_N}.
		Then, let the assumptions of the previous Proposition hold and
		denote by $Y_t$ a random variable with distribution $g_{t}$.
		Knowing $Y_t$, we have $h_{t,K} \overset{\mathcal{L}}{=} \Bin\left(K,Y_t\right)$.
		In particular, the mean synaptic current is given by
		$\esp\left(h_{t,K}\right) = K\esp\left(Y_t\right) = K\proba\left( J_t^{1j} = 1 \right) $.
		Indeed, $\esp\left(h_{t,K}\right)$ is the mean number of strong synapses
		$J_t^{1j}$ (with $j$ such that $\xi_0^j = 1$) at time $t$.
	\end{remarque}
	Finally, we show that $\mathcal{R}$ is contracting and characterises $\pi_K$.
	\begin{proposition}\label{prop-R-contracting} 
		The application $\mathcal{R}$ acting on $F_{[0,1]}$ is contracting for the norm \({L^1(0,1)}\).
		Moreover, there exists a unique \(G^*\in F_{[0,1]}\) invariant for $\mathcal{R}$. 
	\end{proposition} 
	Propositions~\ref{prop:dyn-h} and \ref{prop-R-contracting}
	are proved in the Appendices~\ref{app:prop-dyn-h} and~\ref{app-R-contracting}.
	\begin{corollaire}\label{cor:mes-inv-binmix}
		Let \(G^*\) be the unique fixed point of \(\mathcal{R}\) and \(g^*\) its associated distribution.
		The invariant measure \(\pi_K\) of the Markov chain \(h_{t,K}\) satisfies $\pi_K = \BinMix(K,g^*)$.
		In addition, the invariant measure \(\pi_\infty\) of \(h_t\) is given by 
		$\pi_\infty = \BinoMix{\hat{K},g^*}$, where \(\hat{K}\) has a Binomial law with parameters \(N\)
		and \(f_N\), the two random variables being independent.
		Moreover, the smallest interval $[m_\infty, M_\infty]$ containing the support
		of $g^*$ verifies
		\begin{align}\label{eq:sup-g-star}
		Supp(g^*)\subset\left[0,\frac{f_N q^+}{f_N q^+ + (1-f_N) q^-_{10}}\right]
		:=\left[m_\infty, M_{\infty}\right].
		\end{align}
	\end{corollaire}
	\begin{proof} 
		Let $g^*\in \mathcal{P}([0,1])$ be a probability distribution such that its cumulative
		distribution function $G^*$ satisfies $\mathcal{R}(G^*) = G^*$.
		Then, by Proposition~\ref{prop:dyn-h}, $\BinMix(K,g^*)$ is invariant
		for $\left(h_{t,K}\right)_{t\geq 1}$.
		The result on $\pi_\infty$ follows from Corollary \ref{prop-conv-mes-inv}.
		
		Now, let $[m_{\infty}, M_{\infty}]$ be the convex envelop of the support of $g^*$,
		then $Supp(g^*) \subset [m_{\infty}, M_{\infty}]\subset [0,1]$.
		Thus by the equation \(\mathcal{R}(G^*) = G^*\), we get
		\begin{align*}		
		&m_{\infty} = m_{\infty}(1-f_N q^-_{01})\wedge
		\left(m_{\infty}\left(1-(1-f_N) q^-_{10}-f_N q^+\right)+f_N q^+\right),\\
		&M_{\infty} = 
		M_{\infty}(1-f_N q^-_{01})\vee
		\left(M_{\infty}\left(1-(1-f_N) q^-_{10}-f_N q^+\right)+f_N q^+\right).
		\end{align*}
		As $(1-f_N q^-_{01}) < 1$, the first equation implies that
		$0\leq m_{\infty} \leq  m_{\infty}(1-f_N q^-_{01})$ so $m_{\infty}=0$, and the second equation implies that $M_{\infty} = M_{\infty}\left(1-(1-f_N) q^-_{10}-f_N q^+\right)+f_N q^+$,
		thus $M_{\infty}=\frac{f_N q^+}{f_N q^+ + (1-f_N) q^-_{10}}$.
	\end{proof}
	\begin{remarque}
		Propositions~\ref{prop:dyn-h}, ~\ref{prop-R-contracting}, and
		the first part of Corollary~\ref{cor:mes-inv-binmix} are in~\cite{amit_precise_2010}
		with $q_{10}^-=0$ and $r=1$. We prove them here with a different method.
	\end{remarque}
	\subsection{Main results}
	The learning and the forgetting phases are both described by Markov chains.
	We first give the spectrum of the transition matrices associated to these chains
	and then we give our main results on $t_*$.
	
	\subsubsection*{Spectrum}
	Let $P_{y,K}$ be the transition matrix of the synaptic current
	$\left(h_{t,K}^y\right)_{-r+1 < t \leq 1}$.
	We denote by $\nu_{t,K}^y =\left[\nu_{t,K}^y(0),\ \nu_{t,K}^y(1),\ \hdots,\ \nu_{t,K}^y(K)\right]$
	the distribution of $h_{t,K}^y$.
	We can then write
	\(
	\nu_{1,K}^y = \nu_{0,K}^y P_{y,K} =\nu_{-r+1,K}^y\left(P_{y,K}\right)^r.
	\)
	\begin{proposition}\label{prop-spec-learning} 
		The spectra of $P_{0,K}$ and $P_{1,K}$ are
		\[
		\Sigma\left(P_{0,K}\right)=\left\{\left(1-q_{01}^-\right)^i,\ 0\leq i\leq K\right\} \ \ \text{and} \ \ \Sigma\left(P_{1,K}\right)=\left\{\left(1-q^+\right)^i,\ 0\leq i\leq K\right\}.
		\]
	\end{proposition}
	\begin{proof}
		The dynamics give for all $j>i$, $P_{0,K}^{ij} = P_{1,K}^{ji}= 0$.
		So the matrices are triangular. Their spectra are given by the diagonal elements:
		\[
		\forall i, \quad
		P_{0,K}^{ii} = (1-q_{01}^-)^i
		\quad \text{and} \quad
		P_{1,K}^{ii} = (1-q^+)^i.
		\] 
	\end{proof}
	
	\begin{proposition}\label{prop-spec} 
		The spectrum of the transition matrix $\pk$ of $\left(h_{t,K}\right)_{t \geq 1}$ is 
		\[
		\Sigma\left(\pk\right)=\left\{(1-f_N)\left(1-f_Nq_{01}^-\right)^i
		+f_N\left(1-(1-f_N)q_{10}^--f_Nq^+\right)^i,\ 0\leq i\leq K\right\}.
		\]
	\end{proposition}
	In the following, we denote by $\Lambda_0 = 1-f_Nq_{01}^-$,
	$\Lambda_1 = 1-(1-f_N)q_{10}^--f_Nq^+$ and 
	\[
	\forall i\in \llbracket 0,N \rrbracket,
	\qquad \lambda_i=(1-f_N)\Lambda_0^i + f_N\Lambda_1^i.
	\]
	We prove the previous proposition using the
	\begin{lem}\label{lem-mom-g}
		Let $X$ and $Y$ be two random variables in $[0,1]$ with cumulative distribution
		functions $G_X$ and $G_Y$. 
		We assume that there exist $\eta \in [0,1]$, $a, \bar{a} \in [0,1)$ and  $b, \bar{b} \in(0,1]$
		with $a+b \leq 1$, \(\bar{a}+\bar{b} \leq 1\) such that
		\begin{equation}\label{eq:trans_lineaire}
		G_Y(u) = \eta G_X\left(\frac{u-a}{b}\right)
		+ (1 - \eta) G_X\left(\frac{u-\bar{a}}{\bar{b}}\right).
		\end{equation}
		Then $\forall k \in \nat,\ \esp\left[Y^k\right] = \eta \esp\left[(a + bX)^k\right] + 
		(1 - \eta)\esp\left[(\bar{a} + \bar{b}X)^k\right].$
	\end{lem}
	\begin{proof}
		First, note that \(G_X\left(\frac{u-a}{b}\right)\) is the cumulative distribution function
		of \(a + b X\).
		Second, for all random variables \(U,V, W\), we have 
		\[
		G_U(z) = \eta G_{V}(z) + (1-\eta)G_{W}(z)\Longrightarrow
		\esp[U^k] = \eta \esp[V^k] + (1-\eta)\esp[W^k].
		\]
		This last result is obtained by differentiation, multiplication by $z^k$ and integration.
		It ends the proof of the lemma.
	\end{proof}
	In the proof below, we use the classical conventions
	$
	\displaystyle \binom{i}{j}=0 \text{ when } j>i \text{ or } j<0.
	$
	\begin{proof}[Proof of Proposition~\ref{prop-spec}]
		We denote by $\nu_{t,K} =\left[\nu_{t,K}(0),\ \nu_{t,K}(1),\ \hdots,\ \nu_{t,K}(K)\right]$
		the distribution of $h_{t,K}$.
		Its transition matrix $\pk=\left(P_K^{ij}\right)_{0 \leq i,j \leq K}$ can be
		derived from Proposition~\ref{prop-Markov}:
		\begin{align*}
		&P_K^{ij}
		= (1-f_N) \displaystyle \binom{i}{i-j}(f_Nq_{01}^-)^{i-j}(1-f_Nq_{01}^-)^{j}
		\\
		&+
		f_N \sum_{l=0}^{i}\displaystyle \binom{i}{l}((1-f_N)q_{10}^-)^{l}(\Lambda_1 + f_Nq^+)^{i-l}
		\displaystyle \binom{K-i}{j-i+l}(f_Nq^+)^{j-i+l}(1-f_Nq^+)^{K-j-l}.
		\end{align*}
		Let us define the two matrices $\tilde{P}_K$ and $\qk$ such that for all $0 \leq i,j \leq K$:
		\[
		\tilde{P}_K^{ij}=
		f_N\displaystyle \binom{j}{i}\Lambda_1^{i}(f_Nq^+)^{j-i}
		+(1-f_N)\delta_{ij}\Lambda_0^{i}
		\quad \text{and} \quad
		Q_K^{ij} = \displaystyle \binom{K}{i}\displaystyle \binom{i}{j}(-1)^{i-j}.
		\]
		Then, assuming that $\nu_{t,K} \overset{\mathcal{L}}{=} \BinMix(K,g_t)$ and denoting by
		$U_{t} = \left[U_{t}^0,\ U_{t}^1,\ \hdots,\ U_{t}^K\right]$
		with $U_{t}^k = \int u^k g_t(du)$, we get by definition~\ref{def-bin-mix}:
		$ \nu_{t,K} = U_{t} Q_K.$
		Moreover, by Lemma~\ref{lem-mom-g} we have $U_{t+1} = U_{t} \tilde{P}_K$.
		Finally, by definition we have $\nu_{t+1,K} = \nu_{t,K} P_K$, so we obtain:
		\[
		\nu_{t+1,K} = U_{t+1}Q_K = U_{t} \tilde{P}_K Q_K
		= U_{t} Q_K Q_K^{-1} \tilde{P}_K Q_K 
		= \nu_{t,K} Q_K^{-1} \tilde{P}_K Q_K  = \nu_{t,K} P_K.
		\]
		A straightforward computation shows that $ Q_K P_K = \tilde{P}_K Q_K$.
		Thus $P_K$ and $\tilde{P}_K$ have the same spectrum.
		Finally, $\tilde{P}_K$ is a triangular matrix with $\lambda_i$ as diagonal elements.
	\end{proof}
	We deduce from Proposition~\ref{prop-spec} the rate of convergence of the law of \(h_{t,K}\).
	\begin{corollaire}\label{cor:conv-inv-binmix} 
		For all $0\leq K\leq N$, the sequence of the distributions 
		of the synaptic currents, $\left(\nu_{t,K}\right)_{t\geq1}$, converges exponentially
		fast to the unique invariant measure $\pi_K$.
		In particular, there exists $c\in \reels^+$ such that the distance in total
		variation between $\nu_{t,K}$ and $\pi_K$ satisfies:
		\[ 
		\forall t\geq 1,\ \ ||\nu_{t,K}-\pi_K||_{TV} 
		:= \frac{1}{2}\sum\limits_{l=0}^K |\nu_{t,K}(l)-\pi_K(l)|
		\leq c \lambda_1^{t}.
		\]
	\end{corollaire}
	We discuss, in the second paragraph of Section~\ref{sec:discussion},
	the role played by this eigenvalue $\lambda_1$ in our main results.
	
	\subsubsection*{Memory lifetime}
	Under Assumption~\ref{ass-main:init-cond}, $h_{-r+1,K}$ follows its invariant distribution $\pi_K$,
	a Binomial mixture by Corollary~\ref{cor:mes-inv-binmix}. 
	Thus, by Proposition~\ref{prop:dyn-h}, the processes $(h_{t,K}^y)_{t\geq 1}$ follow also
	Binomial mixtures.
	Combining the inequality provided by Lemma~\ref{lem:mass_move_right},
	inequalities on Binomial tails (Lemma~\ref{lem:tail-bin})
	and a control on the tail of the mixing distribution $g^*$ and on the support of $g_t^1$,
	we prove Theorem~\ref{theo:main-result}.
	\begin{lem}\label{lem:mass_move_right}
		Under Assumption~\ref{ass-main:init-cond}, for all $\theta\in \llbracket 0,N \rrbracket$,
		$\Proba{h_t^0 >\theta} \leq \Proba{\pi_\infty >\theta}$.
	\end{lem}
	\begin{proof}
		The proof is recursive and relies on the functional equation for the cumulative distribution of the synaptic currents~\eqref{eq-cdf-H-2} under Assumption~\ref{ass-main:init-cond}.
		From \eqref{eq-cdf-H-1}, we have for all $x\in[0,1]$, $G_1^0(x) = G^*\left(\frac{x}{\Lambda_0}\right) \geq G^*(x)$.
		Then, 
		\begin{align*}
		G_2^0(x) &= f_N G_1^0\left(\frac{x-f_Nq^+}{\Lambda_1}\right) + (1-f_N) G_1^0\left(\frac{x}{\Lambda_0}\right)
		\\
		&\geq f_N G^*\left(\frac{x-f_Nq^+}{\Lambda_1}\right) + (1-f_N) G^*\left(\frac{x}{\Lambda_0}\right)
		= G^*(x),
		\end{align*}
		and so forth so that for all $t\geq 1$ and $x$, $G_t^0(x) \geq G^*(x)$. It implies that for all
		$K,\ \theta \in \nat$, $\Proba{\BinMix(K,g_t^0)>\theta} \leq \Proba{\BinMix(K,g^*)>\theta}$, which ends the proof.
	\end{proof}
	\begin{lem}\label{lem:tail-bin}
		Let $S_N\overset{\mathcal{L}}{=}\Bin(N,p)$. Then, for all $\varepsilon \in (0,1)$
		\begin{align}
		\proba\left(S_N \geq Np(1+\varepsilon)\right)
		&\leq \exp\left(\frac{-Np\ \varepsilon^2}{2+\varepsilon}\right),\label{eq:ineq-bin-pos}\\
		\proba\left(S_N \leq Np(1-\varepsilon)\right)
		&\leq \exp\left(\frac{-Np\ \varepsilon^2}{2}\right).\label{eq:ineq-bin-neg}
		\end{align}
	\end{lem}
	This Lemma is proved in~\ref{app-proof-lemmas}.
	
	We now give our main results.
	\begin{theoreme}\label{theo:main-result}		
		For $y\in \{0,1\}$, let $(h_t^y)_{t\geq 1}$ be the solutions of \eqref{def-h-t0-Bin}
		with $\xi_0^1=y$ and \eqref{def-h-Bin}.
		Let us assume that Assumptions~\ref{ass-main:init-cond} and~\ref{ass-main:on-f} hold and
		that $q_{01}^-$ and $q^+$ are fixed in $(0,1]$ and $q_{10}^-$ in $[0,1]$. 
		
		Then, for all $0<\delta<1$ and $r\in \nat^*$,
		there exists $N(\delta,r)\in \nat$ such that for all
		$N\geq N(\delta,r)$, there exist \(\theta_{\delta,N} \in \llbracket0,N\rrbracket \) and
		$\hat{t}(\delta,r,N)$ such that for all
		$1 \leq t \leq \hat{t}(\delta,r,N)$,		
		\[
		\proba\left(h_t^0>\theta_{\delta,N}\right)
		\vee
		\proba\left(h_t^1\leq\theta_{\delta,N}\right)
		\leq \delta.
		\]
	\end{theoreme}
	In particular, we have $t_*(\delta,r,N) \geq \hat{t}(\delta,r,N)$.
	This result relies on the study of the mixing distributions $g^*$ and $g_t^1$.
	Thanks to Lemma~\ref{lem:mass_move_right}, we know that as long as $g_t^1$ is far enough from $g^*$,
	the probability of error,
	
	\(
	\proba\left(h_t^0>\theta\right) \vee \ \proba\left(h_t^1\leq\theta\right)
	\leq \proba\left(\BinMix(K,g^*)>\theta\right) \vee \ \proba\left(\BinMix(K,g_t^1)\leq\theta\right),
	\)
	
	\noindent is small enough.
	This condition appears as an inequality depending both on the time and the accepted error $\delta$.
	As long as this inequality holds, there exists a threshold \(\theta\) such that
	the probability of error is below $\delta$ for all previous times.
	\begin{example}
		We give in Remark~\ref{rem:main-result_q_fix} an explicit formula for the lower bound
		$\hat{t}$ on $t_*$ for any couple $(\delta,r)$. We give here a detailed result for a particular choice of parameters.
		Let $q^+ = q_{01}^- = 1$, $q_{10}^-$ small enough, and  $f_N = \frac{q_{10}^-}{3+q_{10}^-}$.
		Explicit computations give
		\[
		\hat{t}(\delta,r,N)=\Bigg\lfloor\frac{\log\left(\frac{1}{9}\right)
			\ \vee\ 
			\log\left(\frac{\sqrt{-2\log(\frac{\delta}{2})Nf_N}-16\log(\frac{\delta}{2})}
			{3Nf_N}\right)}
		{\log\left(1-4f_N\right)}\Bigg\rfloor.
		\]
		For instance, for $q_{10}^- = 0.005$ we get $f_N = 0.00167$ and
		
		\(
		\qquad \qquad
		\centering
		\hat{t}(\delta = 0.001,r = 1, N = 2.10^5) = 246
		\quad \text{and} \quad
		\theta_{\delta,N} = 122.
		\)
		
		\noindent We also give a formula when the depression probabilities depend on $N$ in
		Theorem~\ref{theo:main_result_q_N}.
	\end{example}
	\begin{proof}[Proof of Theorem~\ref{theo:main-result}]
		The proof follows these lines: from Lemma~\ref{lem:mass_move_right} we have that
		$\Proba{h_t^0 > \theta} \leq \pi_\infty\left(]\theta,+\infty[ \right)$. Hence, we propose a threshold $\theta$ based on the measure $\pi_\infty$ such that
		$
		\pi_\infty\left(]\theta,+\infty[ \right) \leq \delta
		$
		and then we bound the time before which $\Proba{h_t^1 \leq \theta} \geq \delta$.
		
		We split $\pi_\infty\left(]\theta,+\infty[ \right)$ in two terms.
		We recall that $\pi_\infty = \BinoMix{K,g^*}$ with $K \overset{\mathcal{L}}{=} \Bino{N,f_N}$
		and $[0,M_\infty]$ is the smallest interval containing the support of $g^*$. So
		\begin{align*}
		\pi_\infty\left(]\theta,+\infty[ \right)
		&= \int_{0}^{M_\infty}\Proba{\Bino{K,u} > \theta}g^*(du)
		= \int_{0}^{M_\infty}\Proba{\Bino{N,f_N u} > \theta}g^*(du)
		\\
		&\leq \Proba{\Bino{N,f_N M_{\delta}} > \theta}
		+ \int_{M_{\delta}}^{M_\infty}g^*(du).
		\end{align*}
		The second equality comes from the following property: assume \(		K\overset{\mathcal{L}}{=}\Bino{N,f_N}\) and conditionally on \(K\), \(X\) is independent of \(K\) with law \(\Bino{K,p}\), then \(X \overset{\mathcal{L}}{=} \Bino{N,f_N p}\).
		Let $Y^*$ be a random variable with distribution $g^*$.
		We propose a value for $M_{\delta}$ using the Bienaym\'e-Tchebytchev inequality:
		\[
		M_{\delta} = \left(\sqrt{\frac{2\text{Var}\left(Y_*\right)}{\delta}}+\Esp{Y_*}\right)
		\wedge M_\infty
		\quad \Rightarrow \quad 
		\int_{M_{\delta}}^{M_\infty}g^*(du) \leq \frac{\delta}{2}.
		\]
		We first fix $\theta_{\delta,N}$ such that
		$\Proba{\Bino{N,f_N M_{\delta}} \geq \theta_{\delta,N} + 1} \leq \frac{\delta}{2}$.
		To do so we apply Lemma~\ref{lem:tail-bin} with
		$\varepsilon = \frac{\theta_{\delta,N} + 1}{N f_N M_{\delta}} - 1$ and obtain:
		\[
		\theta_{\delta,N} = \bigg\lfloor N f_N M_{\delta}
		+ \sqrt{-2\log\left(\frac{\delta}{2}\right)N f_N M_{\delta}}
		-\log\left(\frac{\delta}{2}\right)\bigg\rfloor.
		\]
		
		We now bound the probability of error $\Proba{h_t^1 \leq \theta_{\delta,N}}$.
		Let $[m_t^1,M_t^1]$ be the smallest interval containing the support of $g_t^1$.
		Then, we get:
		\[
		\Proba{h_t^1 \leq \theta_{\delta,N}}
		=\int_{m_t^1}^{M_t^1}\Proba{\Bino{K,u} \leq \theta_{\delta,N}}g_t^1(du)
		\leq \Proba{\Bino{N,f_N m_t^1} \leq \theta_{\delta,N}}.
		\]
		Using Lemma~\ref{lem:tail-bin} with $\varepsilon = 1 - \frac{\theta_{\delta,N}}{N f_N m_t^1}$, we get
		\begin{align*}\label{eq:bin_tail-1}
		\Proba{\Bino{N,f_N m_t^1} \leq \theta_{\delta,N}}
		\leq \exp\left(-\frac{\left(N f_N m_t^1 - \theta_{\delta,N}\right)^2}{2 N f_N m_t^1}\right).
		\end{align*}
		Using the inequality $\sqrt{x} + \sqrt{y} \geq \sqrt{x+y}$ for all $x, y>0$, we obtain that if
		\begin{equation}\label{eq:ineq_m1}
		N f_N m_t^1 \geq \theta_{\delta,N} + \sqrt{-2\log(\delta)\theta_{\delta,N}} -2\log(\delta)
		\end{equation}
		then $\Proba{h_t^1 \leq \theta_{\delta,N}} \leq \delta $.
		Let us define $m_{\delta,N} := \frac{1}{Nf_N}\left(\theta_{\delta,N} + \sqrt{-2\log(\delta)\theta_{\delta,N}} -2\log(\delta)\right)$.
		Using the bound
		$\theta_{\delta,N} \leq \left(\sqrt{Nf_NM_{\delta}}
		+ \frac{\sqrt{-2\log(\frac{\delta}{2})}}{2}\right)^2$ we get
		\begin{equation}\label{eq:ineq_m_delta_N}
		Nf_Nm_{\delta,N} + \frac{3}{2}\log(\delta) = \left(\sqrt{\theta_{\delta,N}} + \frac{\sqrt{-2\log(\delta)}}{2}\right)^2
		\leq \left(\sqrt{M_{\delta}N f_N} +\sqrt{-2\log(\frac{\delta}{2})}\right)^2.
		\qquad \qquad \qquad \qquad \qquad \quad
		\end{equation}
		We now find $m_t^1$.
		From equation~\eqref{eq-cdf-H-1} and the definition of $\mathcal{R}$ (see Notation~\ref{func_eq}),
		we have
		\begin{equation*}
		\forall t \geq 1, \qquad m_{t+1}^1=m_t^1 \Lambda_0\ \wedge\ (m_t^1 \Lambda_1 + f_N q^+).
		\end{equation*}
		We note that for $N$ large enough such that
		$m_t^1 > \frac{f_N q^+}{\Lambda_0-\Lambda_1} \geq M_\infty$, we have
		\[
		\frac{f_N q^+}{1 - \Lambda_1}= M_\infty < m_t^1 \Lambda_1 + f_N q^+ < m_t^1 \Lambda_0 < m_t^1.
		\]
		Denoting by $t_c = \inf\{t \in \nat^*, m_t^1 \leq \frac{f_N q^+}{\Lambda_0-\Lambda_1}\}$, we obtain
		\begin{equation}\label{eq:dynamique_m_t1}
		m_{t}^1 = \left(\left(m_1^1 - M_\infty\right)\Lambda_1^{(t \wedge t_c) - 1} + M_\infty\right)
		\Lambda_0^{(t - t_c)\mathbbm{1}_{t>t_c}}.
		\end{equation}
		Let us now consider $q^+$, $q_{01}^-$ and $q_{10}^-$ fixed in $(0,1]$.
		By definition, $M_{\delta} \leq M_{\infty}$, hence
		\[
		Nf_Nm_{\delta,N}
		\leq \left(\sqrt{M_{\infty}N f_N} +\sqrt{-2\log\left(\frac{\delta}{2}\right)}\right)^2-
		\frac{3}{2}\log(\delta).
		\]
		Therefore, the inequality~\eqref{eq:ineq_m1} holds true as long as
		\begin{equation}\label{eq:ineq_finale_q_fix}
		t - 1 \leq \left(\Bigg\lfloor\frac{\log\Big(
			\frac{2 \sqrt{-2\log(\frac{\delta}{2})Nf_N M_\infty}-4\log(\frac{\delta}{2})}
			{Nf_N(m_1^1 - M_\infty)}\Big)}
		{\log(\Lambda_1)}\Bigg\rfloor \wedge t_c\right)_+
		\text{ with } (x)_+ = x \mathbbm{1}_{x \geq 0}.
		\end{equation}
		But $m_1^1 = 1 - (1-q^+)^r \geq q^+$ and both $\frac{f_N q^+}{\Lambda_0-\Lambda_1}$
		and $M_\infty = \frac{f_N q^+}{f_N q^+ + (1 - f_N) (q_{01}^- + q_{10}^-)}$ tends to $0$ with increasing $N$
		so there exists $N(\delta,r)$ such that for all $N \geq N(\delta,r)$
		we can remove ``$(\ )_+$'' in the inequality~\eqref{eq:ineq_finale_q_fix}:
		that is to say for all $N \geq N(\delta,r)$ such that,
		\[
		\frac{2 \sqrt{-2\log(\frac{\delta}{2})Nf_N M_\infty}-4\log(\frac{\delta}{2})}
		{Nf_N(m_1^1 - M_\infty)} < 1 \quad \text{and} \quad m_1^1 > \frac{f_N q^+}{\Lambda_0-\Lambda_1}.
		\]
		Using the fact that for all $\delta\in (0,1)$,
		$\sqrt{-2\log(\frac{\delta}{2})} \leq -2\log(\frac{\delta}{2})$ and $m_1^1 \geq q^+$ we get
		the two following conditions on $N$:
		\begin{align}\label{eq:condition_N}
		2\exp\left(-\frac{Nf_N(m_1^1 - M_\infty)}
		{4(\sqrt{Nf_NM_\infty} + 1)}\right)<\delta
		\ \text{ and }\
		\frac{f_{N}q^+}{f_{N}q^+ + (1-f_{N})q^-_{10}- f_Nq^-_{01}}< q^+.
		\end{align}
		
		In the particular case $q_{10}^- = 0$, we have $M_\infty = 1$ so the dynamics of $m_t^1$ is simply
		$m_{t}^1 = m_1^1 \Lambda_0^{t - 1}$. We compute $\Esp{Y_*}$ and $\text{Var}\left(Y_*\right)$ using Lemma~\ref{lem-mom-g} and equation~\eqref{eq-cdf-H-1}:
		\[
		\Esp{Y_*} = \frac{f_N^2 q^+}{1-\lambda_1} = \frac{f_Nq^+}{f_Nq^+ + (1-f_N)(q_{01}^- + q_{10}^-)}
		,\ \  
		\text{Var}\left(Y_*\right) = \frac{f_N^5 (1-f_N) {q^+}^2 {q_{01}^-}^2}{(1-\lambda_1)^2(1-\lambda_2)}.
		\]
		Hence,
		\begin{equation}\label{eq:M_delta}
		M_\delta = \frac{f_Nq^+\left(1+q_{01}^-\sqrt{\frac{2f_N}{\delta(1-\lambda_2)}}\right)}
		{f_Nq^+ + (1-f_N)(q_{01}^-+q_{10}^-)}.
		\end{equation}
		We note that $1-\lambda_2 \sim_{N\infty} 2f_N(q_{01}^-+q_{10}^-)$, so $M_\delta$ converges to $0$
		with increasing $N$. Thus, by inequality~\eqref{eq:ineq_m_delta_N}, there exists a
		$N(\delta,r)$ such that for all $N \geq N(\delta,r)$, $ m_{\delta,N} < 1$.
		We conclude that for all $N \geq N(\delta,r)$,
		the inequality~\eqref{eq:ineq_m1} holds true as long as
		\begin{equation*}
		t - 1 \leq \Bigg\lfloor\frac{\log\Big(\frac{\left(\sqrt{M_{\delta}N f_N}
				+\sqrt{-2\log(\frac{\delta}{2})}\right)^2 - \frac{3}{2}\log(\delta)}{Nf_N}\Big)}
		{\log(\Lambda_0)}\Bigg\rfloor > 0.
		\end{equation*}
	\end{proof}
	\begin{remarque}\label{rem:main-result_q_fix}
		Recall that 
		$M_{\infty} = \frac{f_Nq^+}{1-\Lambda_1}$, $m_1^1 = 1-(1-q^+)^r$,
		$M_\delta = \frac{f_Nq^+\left(1+q_{01}^-\sqrt{\frac{2f_N}{\delta(1-\lambda_2)}}\right)}
		{f_Nq^+ + (1-f_N)(q_{01}^-+q_{10}^-)}$,
		$\Lambda_0 = 1-f_Nq_{01}^-,\ \Lambda_1 = 1 - f_Nq^+ - (1-f_N)q^-_{10}$ and
		$\lambda_2 = f_N \Lambda_1^2 + (1-f_N)\Lambda_0^2$.
		
		We proved that under Assumptions~\ref{ass-main:init-cond} and~\ref{ass-main:on-f},
		for all $\delta$, $r$, $N\geq N(\delta,r)$
		($N$ for which the two conditions given by~\eqref{eq:condition_N} are satisfied),
		there exists $\theta_{\delta,N} \in \llbracket 0,N \rrbracket$ and $\hat{t}$ such that
		for all $1 \leq t\leq \hat{t}(\delta,r,N)$, $\proba\left(h_t^0>\theta_{\delta,N}\right) \vee\ \proba\left(h_t^1\leq\theta_{\delta,N}\right)\leq \delta.$
		
		In particular, if $q^-_{01},q^-_{10},q^+\in (0,1]$
		\begin{align*}
		\hat{t}(\delta,r,N) - 1 &=
		\Bigg\lfloor\frac{\log\Big(
			\frac{2 \sqrt{-2\log(\frac{\delta}{2})Nf_N M_\infty}-4\log(\frac{\delta}{2})}
			{Nf_N(m_1^1 - M_\infty)}\Big)}
		{\log(\Lambda_1)}\Bigg\rfloor
		\wedge
		{\Bigg\lfloor\frac{\log\left(\frac{f_N^2q^+ q_{01}^-}{(1 - \Lambda_1)(\Lambda_0 - \Lambda_1)(m_1^1-M_{\infty})}\right)}{\log(\Lambda_1)}\Bigg\rfloor},
		\end{align*}
		and if $q^-_{10} = 0$,
		$\hat{t}(\delta,r,N) - 1 = \Big\lfloor\frac{\log\Big(\frac{\left(\sqrt{M_{\delta}N f_N}
				+\sqrt{-2\log(\frac{\delta}{2})}\right)^2 - \frac{3}{2}\log(\delta)}{Nf_N}\Big)}
		{\log(\Lambda_0)}\Big\rfloor$.
	\end{remarque}
	\begin{theoreme}\label{theo:main_result_q_N}
		Assume Assumptions~\ref{ass-main:init-cond}, \ref{ass-main:on-f} and \ref{ass-main:on-q10-q01} are satisfied.
		Then, for all $\delta \in (0,1)$, $r$ large enough, there exists \(N(\delta, r)\in \nat\)
		such that for all \( N\geq N(\delta, r) \),
		\begin{equation*}
		\hat{t}(\delta,r,N) = t_c +
		\bigg\lfloor\frac{\log\left(C(\delta,r,N)\right)}{\log(\Lambda_0)}\bigg\rfloor,
		\end{equation*}
		with $t_c$  defined in~\eqref{eq:dynamique_m_t1} and \(C(\delta,r,N)\in(0,1)\) satisfies
		$\frac{\log\left(C(\delta,r,N)\right)}{f_N} \rightarrow + \infty$.
		Moreover, if $\lim a_N$, and $\lim b_N$ exist and are finite, 
		$\frac{\log\left(C(\delta,r,N)\right)}{\log(\Lambda_0)}$ is on the order of $\frac{1}{f_N^2}$.
	\end{theoreme}
	We note that $\log(\Lambda_0) = \log(1 - a_N f_N^2) \sim_{N\infty} - a_N f_N^2$.
	Concerning $C(\delta,r,N)$ (and then $\hat{t}(\delta,r,N)$),
	it mainly depends on the different large $N$ asymptotic of $a_N$ and $b_N$.
	We detail in Remark~\ref{rem:main-result1} the different large $N$ asymptotic of $\hat{t}(\delta,r,N)$.
	\begin{proof}
		We use the results proved in the  proof of Theorem~\ref{theo:main-result}.
		From the dynamics of $m_t^1$ given by the equation~\eqref{eq:dynamique_m_t1} 
		and the bound $m_{t_c}^1 \geq m_1^1 \wedge M_\infty$,
		we obtain that the inequality~\eqref{eq:ineq_m1} is satisfied as long as
		\begin{equation}\label{eq:ineq_finale}
		t - 1 \leq t_c +
		\left(\Bigg\lfloor\frac{\log\left(\frac{m_{\delta,N}}
			{m_{t_c}^1}\right)}{\log(\Lambda_0)}\Bigg\rfloor\right)_+
		\leq t_c +
		\left(\Bigg\lfloor\frac{\log\left(\frac{m_{\delta,N}}
			{(m_1^1 \wedge M_\infty)}\right)}{\log(\Lambda_0)}\Bigg\rfloor\right)_+.
		\end{equation}
		We can remove ``$(\ )_+$'' in the last inequality if there exists $N_0$ such that
		\[
		\forall N \geq N_0, \qquad \frac{m_{\delta,N}}{(m_1^1 \wedge M_\infty)} < 1.
		\]
		Using the inequality~\eqref{eq:ineq_m_delta_N}, we deduce that this is the case if
		\begin{align}\label{eq:ineq_condition_existence_N0}
		C(\delta,r,N) = \sqrt{\frac{M_{\delta}}{m_1^1 \wedge M_\infty}}
		+2\sqrt{\frac{-\log(\frac{\delta}{2})}{ \left(m_1^1 \wedge M_\infty\right) N f_N}}
		< 1.
		\end{align}
		From the previous computation of $M_\delta$, see equation~\eqref{eq:M_delta}, we obtain
		\[
		\frac{M_{\delta}}{M_{\infty}}
		= \left(1-\frac{(1-f_N)a_N}{q^+ + (1-f_N)(a_N+b_N)}\right)
		\left(1+a_N f_N\sqrt{\frac{2f_N}{\delta(1-\lambda_2)}}\right).
		\]
		Thus, we compare the three terms (recalling that $m_1^1 = 1-(1-q^+)^r$)
		\[
		\frac{a_N}{q^+ + a_N + b_N}, \quad a_N f_N\sqrt{\frac{2f_N}{\delta(1-\lambda_2)}}
		\quad \text{and} \quad
		\frac{-\log(\frac{\delta}{2})}{ \left((1-(1-q^+)^r) \wedge M_\infty\right) N f_N}.
		\]
		First, $(1-\lambda_2) \sim_{N\infty} 2f_N^2(a_N+b_N+q^+)$. Then, we have to separate the different
		cases:
		\begin{itemize}[leftmargin=*]
			\item If $b_N$ tends to $+\infty$,
			both $M_{\delta}$ and $M_\infty$ converge to $0$.
			Hence, $\left(1 - (1-q^+)^r\right) \wedge M_\infty = M_\infty$
			and the Assumption~\ref{ass-main:on-q10-q01}
			, in particular $\lim\limits_{N \infty}
			\ q_{01,N}^- = \lim\limits_{N \infty}\ q_{10,N}^-
			= \lim\limits_{N \infty} \frac{b_N^2}{N f_N a_N}
			= \lim\limits_{N \infty} \frac{b_N}{N f_N}
			= 0$,
			enables us to conclude that if $b_N = \smallO{a_N}$,
			$C(\delta,r,N) \sim_{N\infty} \sqrt{\frac{b_N}{a_N}} + 2\sqrt{\frac{-\log(\frac{\delta}{2})b_N}{q^+ N f_N}} \rightarrow 0$,
			else,
			$C(\delta,r,N) \sim_{N\infty} \left(1-\frac{a_N}{a_N + b_N}\right)$
			so the inequality~\eqref{eq:ineq_condition_existence_N0} holds true for any $r$
			and for a $N$ large enough.
			
			\item If $a_N$ tends to $+\infty$ and not $b_N$,
			then $M_{\delta}$ converges to $0$ and $M_\infty$ converges
			to $1$ (resp. $\frac{q^+}{q^+ + b}$) if $b_N$ converges to $0$ (resp. $b$).
			Thus, $C(\delta,r,N)$ converges to $0$ with large $N$ and for any $r$,
			the inequality~\eqref{eq:ineq_condition_existence_N0} is satisfied.
			
			\item If $a_N$ tends to $0$ and $b_N$ converges to $b>0$,
			then $M_{\delta}$ converges to $M_\infty$. Then, there exists $r_0$
			such that $\left(1 - (1-q^+)^{r_0}\right) \wedge M_\infty = M_\infty$.
			Using the assumption $\lim\limits_{N \infty} a_N N f_N = +\infty$ we have
			for all $r\geq r_0$, $C(\delta,r,N)\sim_{N\infty} \left(1-\frac{a_N}{q^+ + b}\right)$.
			
			\item
			In all other cases, $M_{\delta}$ and $M_\infty$ converges to a value in $(0,1)$.
			Moreover, $M_{\delta} < M_\infty$ so there exists $r_0$ such that for all $r \geq r_0$,
			$\left(1 - (1-q^+)^{r_0}\right) \wedge M_\infty > M_\delta$,
			so $C(\delta,r,N)$ converges in $(0,1)$ with large $N$.
		\end{itemize}
	\end{proof}
	\begin{remarque}\label{rem:main-result1}
		In the large $N$ asymptotic (under Assumptions~\ref{ass-main:on-f} and~\ref{ass-main:on-q10-q01}),
		we can compute the terms equivalent to $\hat{t}$ in the different $a_N$, $b_N$ cases
		($a\in \reels_*^+$ and $b\in \reels^+$):
		\begin{table}[h]
			\begin{center}
				\begin{tabular}{|c|c|}
					\hline 
					conditions on $a_N,\ b_N$ and $r$
					& $\hat{t}(\delta,r,N)$ for large \(N\)
					\\
					\hline
					$b_N \rightarrow +\infty,\  b_N=\smallO{a_N} ,\ \forall r $
					& $\frac{\log\left(\sqrt{\frac{b_N}{a_N}} + 2\sqrt{\frac{-\log(\frac{\delta}{2})b_N}{q^+ N f_N}}\right)}{f_N^2a_N}$
					\\
					\hline
					$a_N,b_N \rightarrow +\infty \text{ of same order},\ \forall r $
					& $-\frac{\log\left(1-\frac{a_N}{a_N + b_N}\right)}{2f_N^2 a_N}$
					\\
					\hline 
					$a_N = \smallO{b_N},\ b_N \rightarrow +\infty,\ \forall r $
					& $\frac{1}{2f_N^2b_N}$
					\\
					\hline 
					$a_N \rightarrow +\infty,\ b_N \rightarrow b\in \reels^+,\ \forall r $
					& $\frac{-\log\left(\sqrt{\frac{q^+}{\big((1-(1-q^+)^r) \wedge \frac{q^+}{q^+ + b}\big)a_N}}+
						2\sqrt{\frac{-\log(\frac{\delta}{2})(q^+ + b)}{q^+ N f_N}}\right)}{f_N^2a_N}$
					\\
					\hline
					$a_N \rightarrow 0,\ b_N \rightarrow b>0,\ \forall r>r_0 $
					& $\frac{1}{2f_N^2(q^+ + b)}$
					\\
					\hline 
					$a_N = a,\ b_N \rightarrow 0 \text{ ou } b_N = 0,\ \forall r>r_0 $
					& $-\frac{\log\left(\frac{q^+}{(1-(1-q^+)^r)(q^+ + a)}\right)}{2f_N^2 a}$
					\\
					\hline 
					$a_N = a,\ b_N = b,\ \forall r>r_0 $
					& $-\frac{\log\left(1-\frac{a}{q^+ + a + b}\right)}{2f_N^2 a}$ \\
					\hline
				\end{tabular}
			\end{center}
			\caption{The large $N$ equivalent of $\hat{t}(\delta,r,N)$
				in function of \(a_N\) and \(b_N\).}\label{fig:table_asympt}
		\end{table}
	\end{remarque}
	\begin{remarque}\label{rem:main-result2}
		Note that we have also proved the following result: 
		
		\noindent For every \(\delta > 0\) and \(N\) large enough, there exists \(r_0\) such that,
		if the initial signal is presented at least \(r_0\) times, then it is \textit{well}
		memorized after at least \(\hat{t}(\delta,r,N)\) presentations of \textit{noisy} signals.
		Moreover, in the large $r$ asymptotic, $h_{1,K}^0\overset{\mathcal{L}}{=} \delta_0$ (dirac in $0$)
		and $h_{1,K}^1\overset{\mathcal{L}}{=} \delta_K$ (dirac in $K$).
		Thus, the initial error is null.
		However, the $\hat{t}$ increases with $r$ until reaching a threshold value
		which is given by the expression of Remarks~\ref{rem:main-result_q_fix} and~\ref{rem:main-result1}
		replacing the quantities $m_1^1$ by $1$.
	\end{remarque}
	\section{Simulations}\label{sec-4}
	\vspace{-0.5em}
	Our code follows these lines. We draw $\xi_0$ and $K = \sum_{j=2}^{N+1} \xi_0^j$.
	We simulate a trajectory of $h_{t,K}$ long enough to be under the invariant measure.
	We perform $r$ presentations of the signal to be learnt and then compute the
	trajectories of $h_{t,K}^y$, $y\in \{0,1\}$.
	We reiterate this procedure $N_{MC} = 10^7$ times to get an approximation of
	the distributions of $h_{t}^y$.
	
	The result of Theorem~\ref{theo:main-result} is interesting for large values of $Nf_N$ (small errors)
	combined with a small $f_N$ (non-negligible $\hat{t}$).
	In this context, we need to compute many trajectories before the synaptic currents cross
	a reasonable threshold $\theta$.
	\begin{figure}[h] 
		\centering 
		\subfloat[]{
			\includegraphics[width=0.45\textwidth]{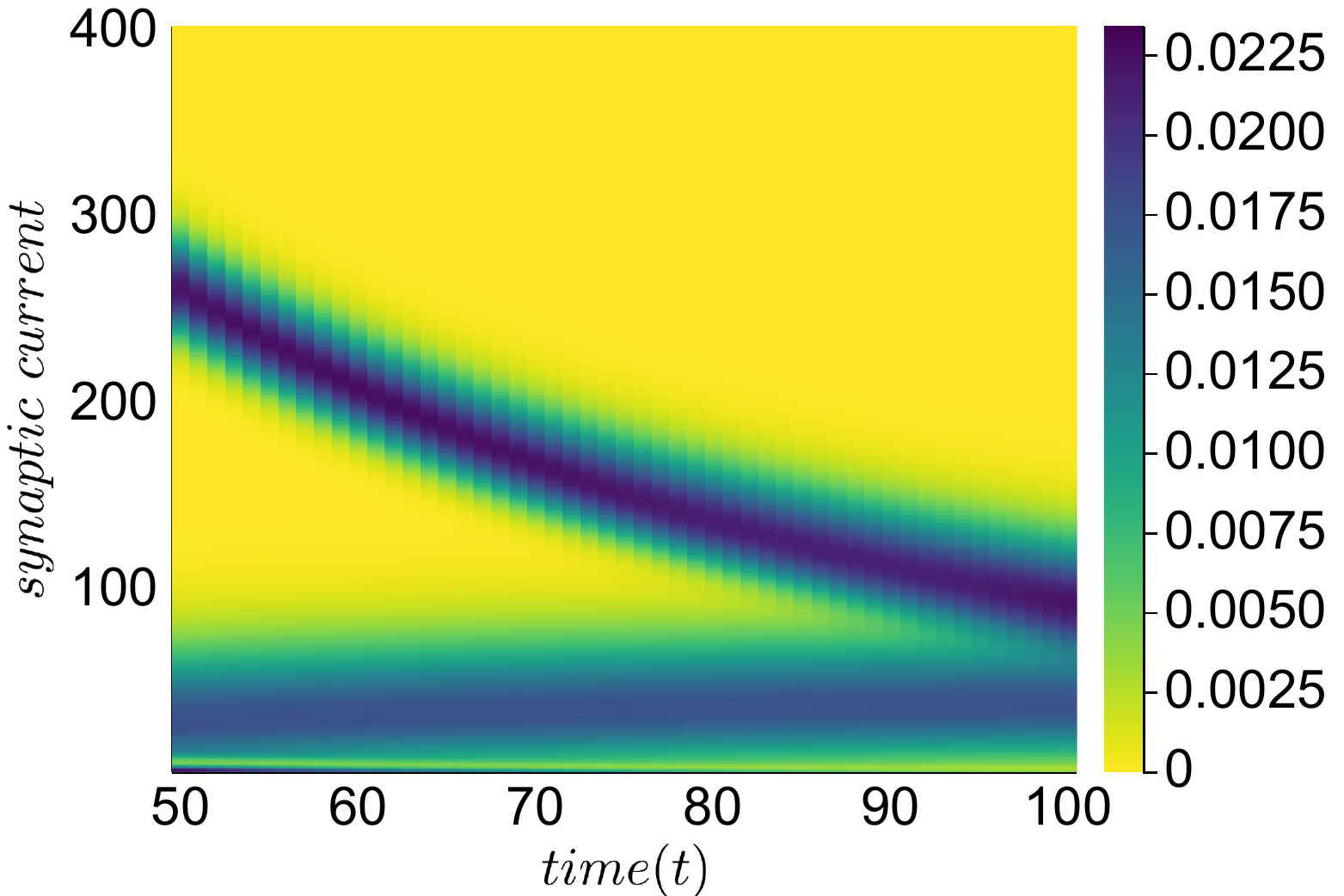} \label{sub:heatmap-f05-qpm5-sup13} }
		\subfloat[]{ \includegraphics[width=0.45\textwidth]{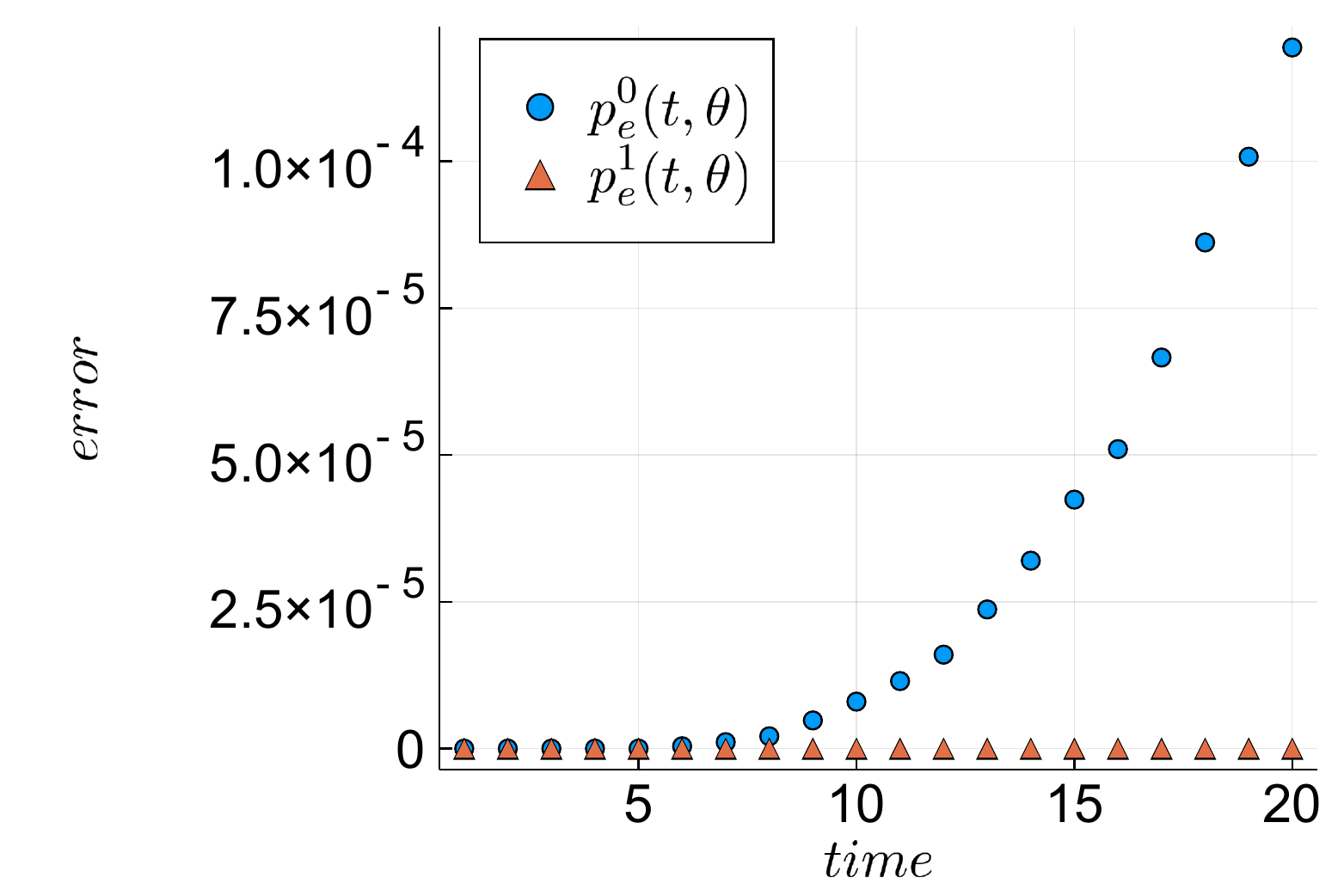} \label{sub:erreur-f05-qpm5-time20} }
		\\ 
		\subfloat[]{ \includegraphics[width=0.45\textwidth]{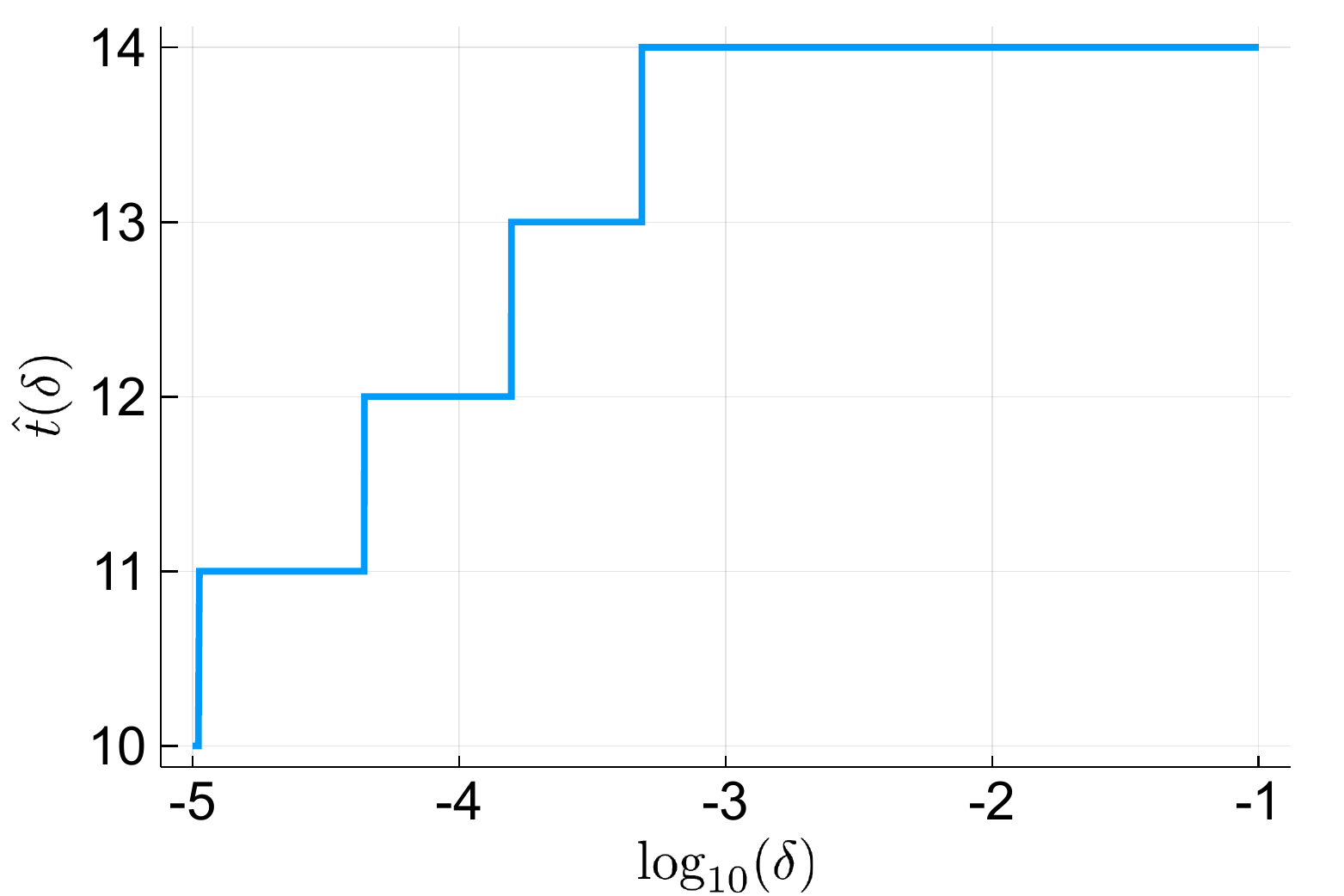}
			\label{sub:t_chapeau-f05-qpm5} } 
		\subfloat[]{
			\includegraphics[width=0.45\textwidth]{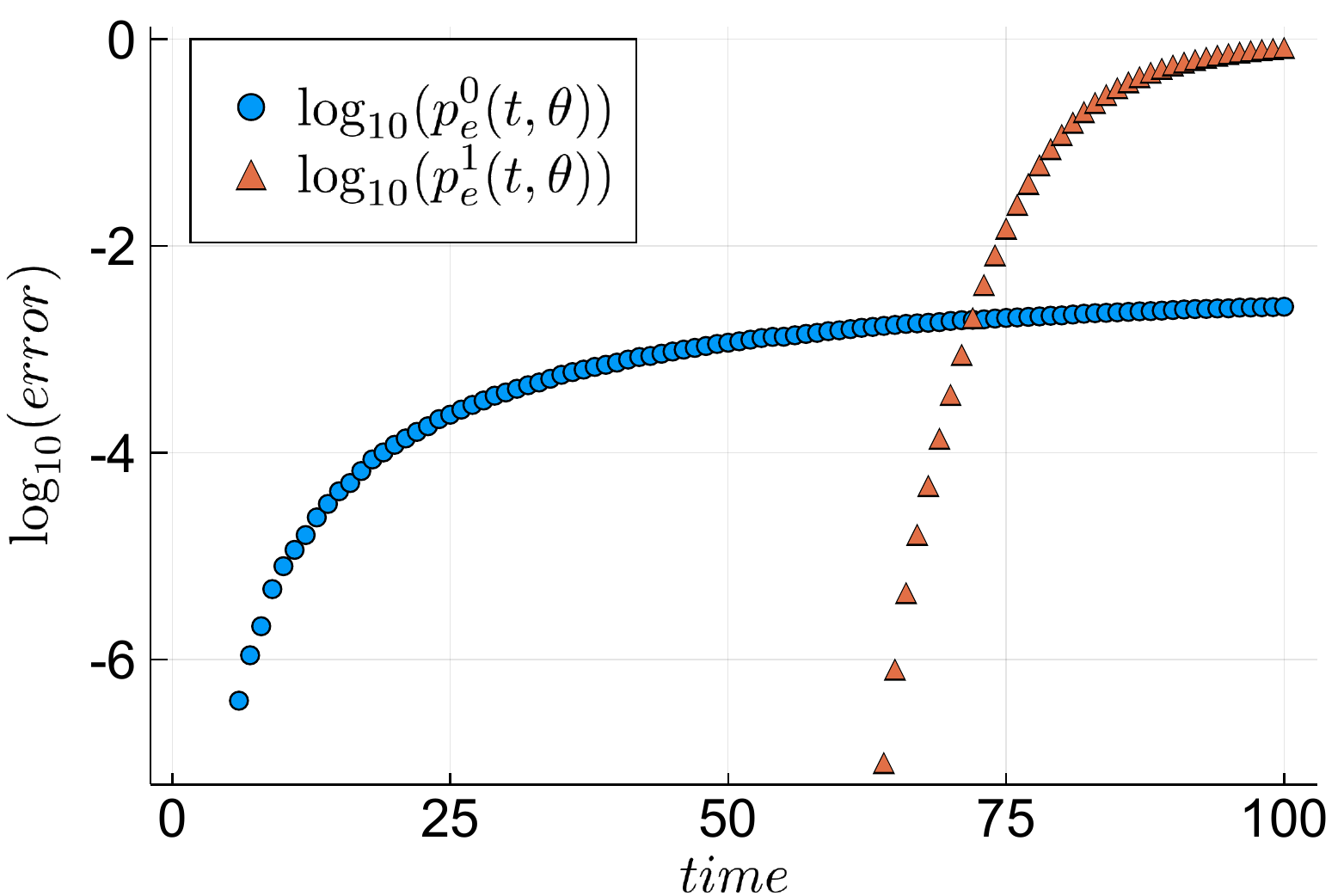} \label{sub:erreur-f05-qpm5-time100} }
		\caption{
			\eqref{sub:heatmap-f05-qpm5-sup13}
			The sum of the distributions of $h_{t,K}^0$ and $h_{t,K}^1$.
			The colour bar gives the probability values.
			\eqref{sub:erreur-f05-qpm5-time20} (resp.~\eqref{sub:erreur-f05-qpm5-time100})
			The numerical errors $p_e^0(t,\theta)$ and $p_e^1(t,\theta)$
			on a short (resp. large) timescale.
			\eqref{sub:t_chapeau-f05-qpm5}
			$\hat{t}$ as a function of $\delta$ on the logarithmic to the base ten scale.
			Parameters: $\theta = 117,\ N = 20\ 000,\ f_N = 0.05,\ 
			q^+ = q^-_{01}= 0.5,\ q^-_{10}= 0.05,\ r=3.$
		}\label{image.simu-result} 
	\end{figure}
	
	In Figure~\ref{sub:heatmap-f05-qpm5-sup13}, the top (resp. bottom) roughly 
	represents the distribution of $h_{t,K}^1$ (resp. $h_{t,K}^0$).
	Before time $t=50$, the distribution of $h_t^0$ is highly concentrated in $0$.
	Indeed, looking carefully to Figure~\ref{sub:heatmap-f05-qpm5-sup13}, we can observe a residue
	of this high probability (dark blue) for very weak synaptic currents
	until time $t=65$, see also Figure~\ref{sub:histogrammes-f05-qpm5-h0}. 
	This concentration drastically reduces the contrast of the plot.
	That is why the time axis starts at $t=50$ in Figure~\ref{sub:heatmap-f05-qpm5-sup13}.
	This figure shows that a threshold $\theta$ around one hundred is a good choice:
	it seems to maximises the time for which the threshold estimation holds true.
	With this threshold, the numerical errors $p_e^0(t,\theta)$ and $p_e^1(t,\theta)$
	does not exceed $10^{-4}$, see Figure~\ref{sub:erreur-f05-qpm5-time20}, before time $15$.
	It is coherent with $\hat{t}$ plotted in Figure~\ref{sub:t_chapeau-f05-qpm5}. 
	Indeed, the time $\hat{t}$ is equal to $12$ for errors on the order of $10^{-4}$, see
	Figure~\ref{sub:t_chapeau-f05-qpm5}. 
	Moreover, in Remark~\ref{rem:main-result1}, the result $\hat{t}$ is a maximum between two times. 
	The second one does not depend on the error $\delta$
	(it is called $t_c$ in the proof of Theorem~\ref{theo:main-result},
	see equation~\eqref{eq:dynamique_m_t1}).
	This explains the plateau starting at an error just before $10^{-3}$ in
	Figure~\ref{sub:t_chapeau-f05-qpm5}.
	Indeed, for this set of parameters and $\delta$ large enough,
	the time $\hat{t}$ is equal to $t_c$.
	Finally, in Figure~\ref{sub:erreur-f05-qpm5-time100}, we note that $p_e^0$
	is above $p_e^1$ for small values of $t$.
	Then, around time $t=70$, $p_e^1$ increases quickly until a value close to one
	whereas $p_e^0$ stays below $10^{-2}$.
	This is because the majority of the mass of the distribution of $h_{t,K}^0$
	stays less than $\theta$.
	On the other hand, most of the mass of the distribution of $h_{t,K}^1$ crosses
	$\theta$ around time \(t=70\).
	So, the error \(p_e^1\) becomes large.
	We present the histograms of the distributions of the synaptic currents at certain times.
	\newpage
	\begin{figure}[h]
		\centering 
		\subfloat[]{\includegraphics[width=0.45\textwidth]{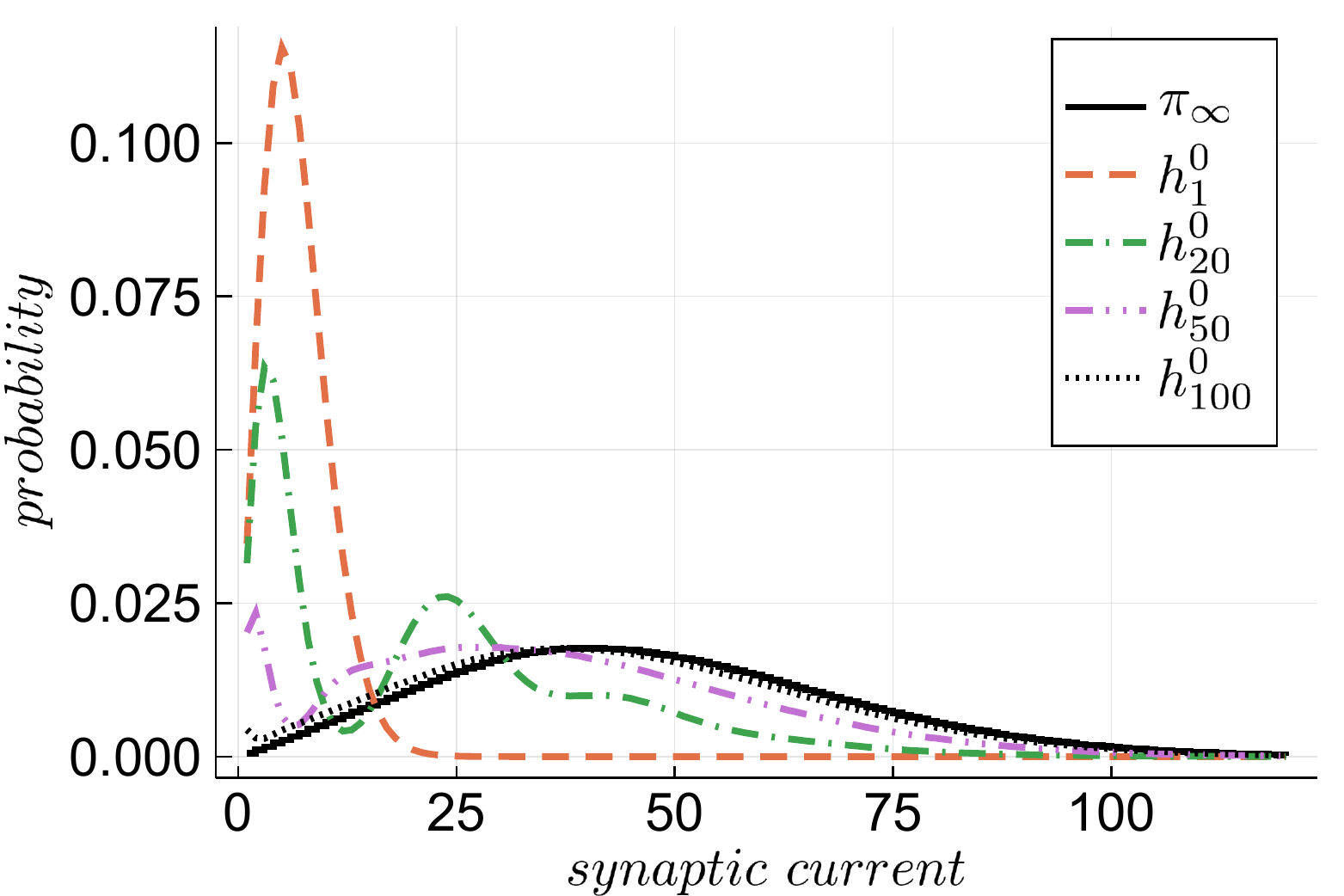} \label{sub:histogrammes-f05-qpm5-h0} }
		\subfloat[]{ \includegraphics[width=0.45\textwidth]{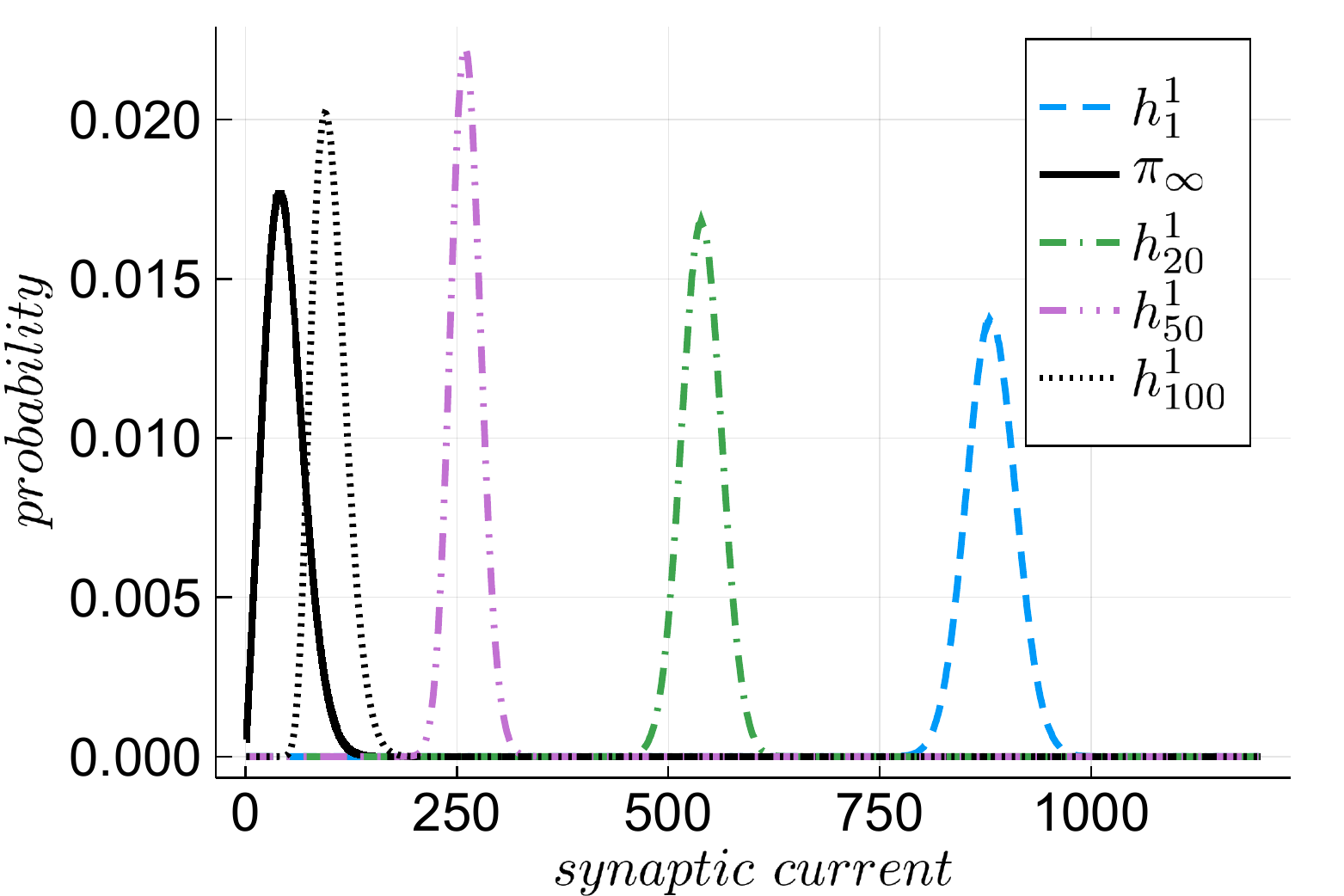} \label{sub:histogrammes-f05-qpm5-h1} }
		\\
		\subfloat[]{\includegraphics[width=0.45\textwidth]{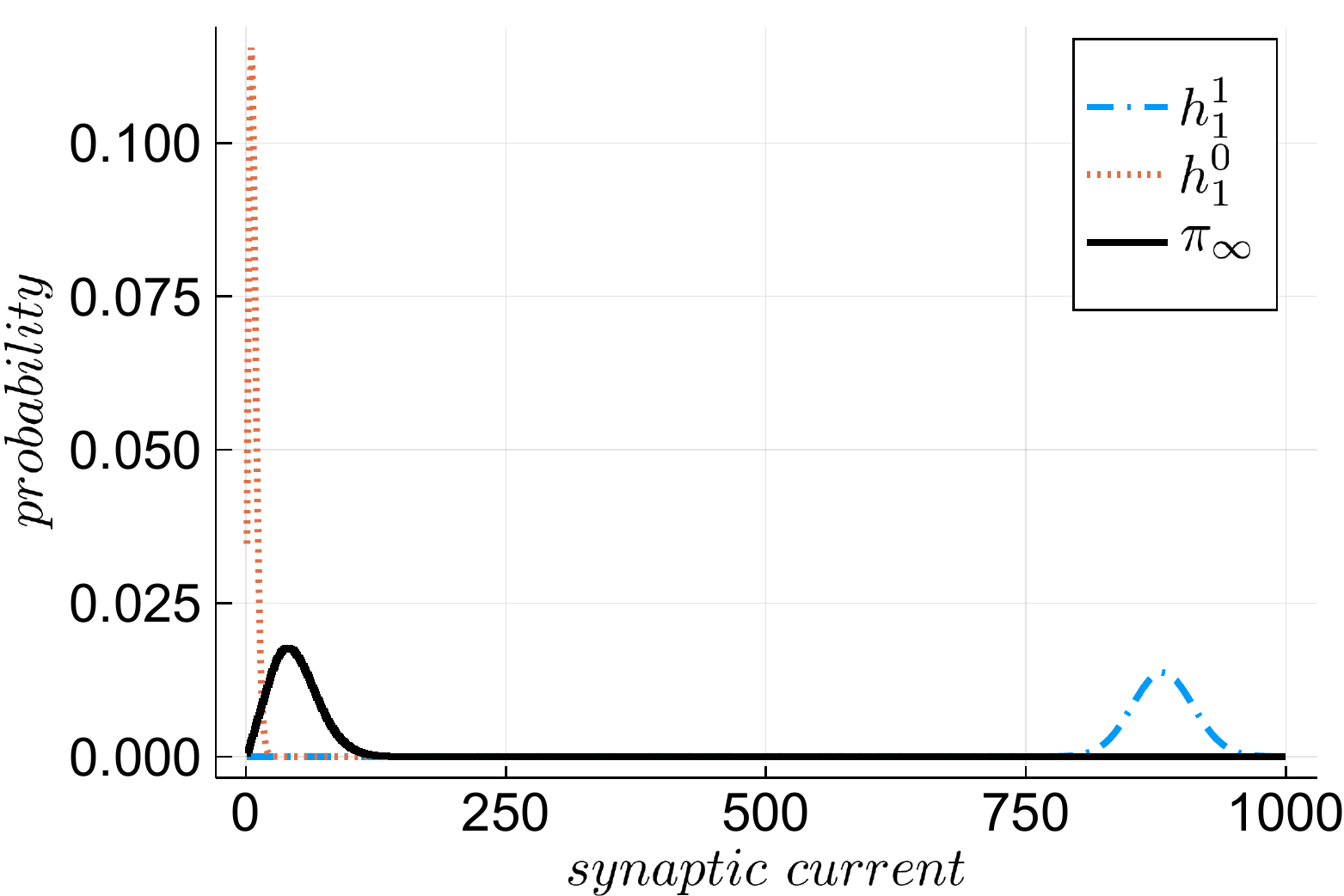} \label{sub:histogrammes-f05-qpm5-t_init} }
		\subfloat[]{ \includegraphics[width=0.45\textwidth]{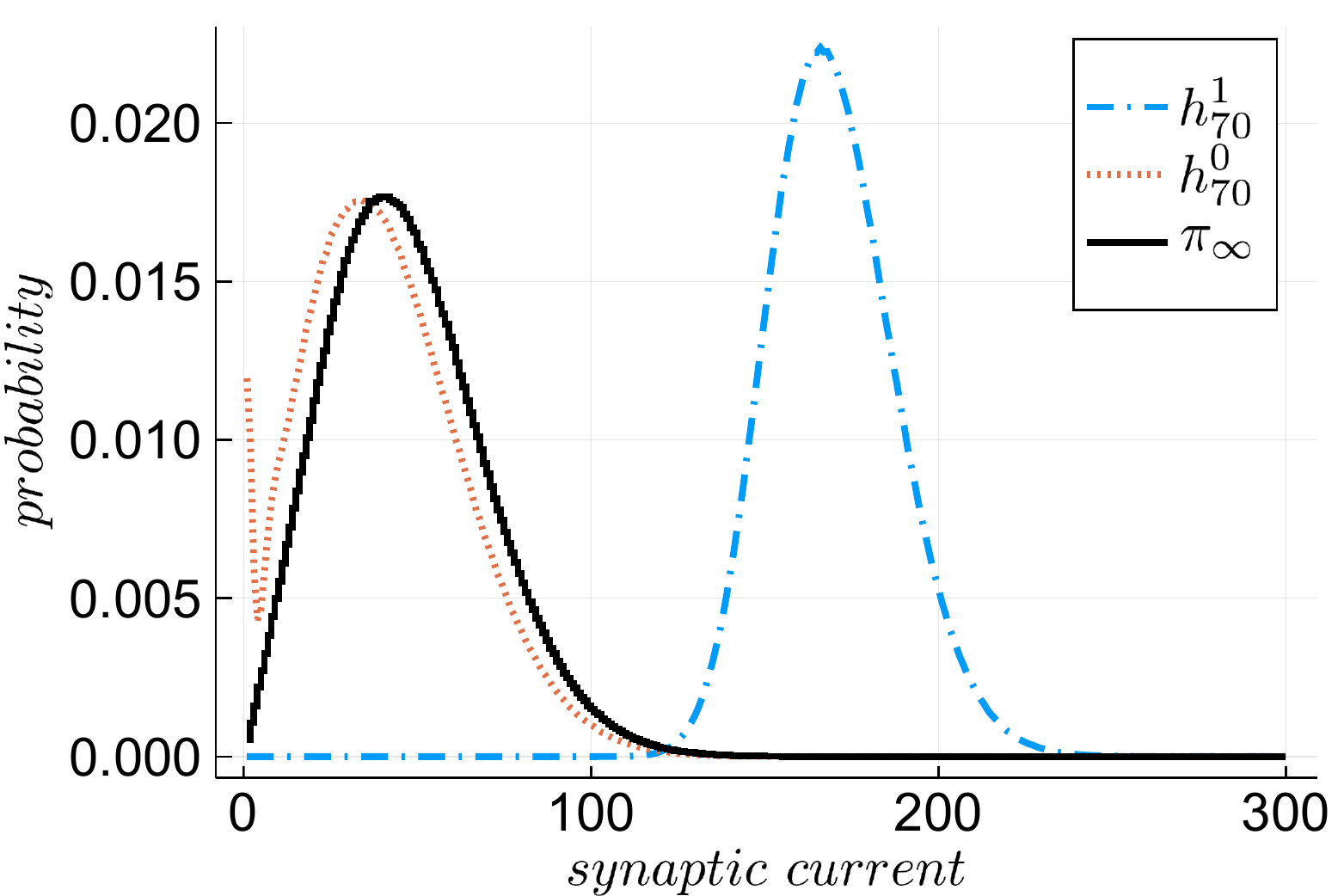} \label{sub:histogrammes-f05-qpm5-t-50-80} }
		\caption{
			\eqref{sub:histogrammes-f05-qpm5-h0}
			Histograms of the distributions of $h_t^0$ at different times.
			\eqref{sub:histogrammes-f05-qpm5-h1}
			Histograms of the distributions of $h_t^1$ at different times.
			\eqref{sub:histogrammes-f05-qpm5-t_init}
			Distributions of $h_1^y$ just after the learning phase and the invariant measure.
			\eqref{sub:histogrammes-f05-qpm5-t-50-80}
			Distributions of $h_t^y$ at $t=70$ and the invariant measure.
			Parameters: $N = 20\ 000,\ f_N = 0.05,\ 
			q^+ = q^-_{01}= 0.5,\ q^-_{10}= 0.05,\ r=3.$
		}\label{image.histo-result}
	\end{figure}
	We note again that the invariant measure is concentrated around small values.
	This enables the post learning distribution of $h_1^0$ to have a small variance,
	see Figures~\ref{sub:histogrammes-f05-qpm5-h0} 
	and~\ref{sub:histogrammes-f05-qpm5-t_init}. 
	However, the variance of this distribution increases quickly.
	In particular, the distribution of $h_t^0$ has a multimodal shape with a
	high proportion of the mass staying near $0$ for more than $50$ presentations after learning. 
	On the other hand, the distribution of $h_t^1$ keeps a unimodal shape with a variance decreasing
	at the beginning, then increasing before decreasing again, see Figure~\ref{sub:histogrammes-f05-qpm5-h1}.
	Distributions stay well separated approximately until time $t=70$, see
	Figure~\ref{sub:histogrammes-f05-qpm5-t-50-80}.
	
	In order to illustrate the role played by the parameter $r$, we plot the distributions just after the
	learning phase for different values of $r$.
	\newpage
	\begin{figure}[h] 
		\centering 
		\subfloat[]{
			\includegraphics[width=0.45\textwidth]{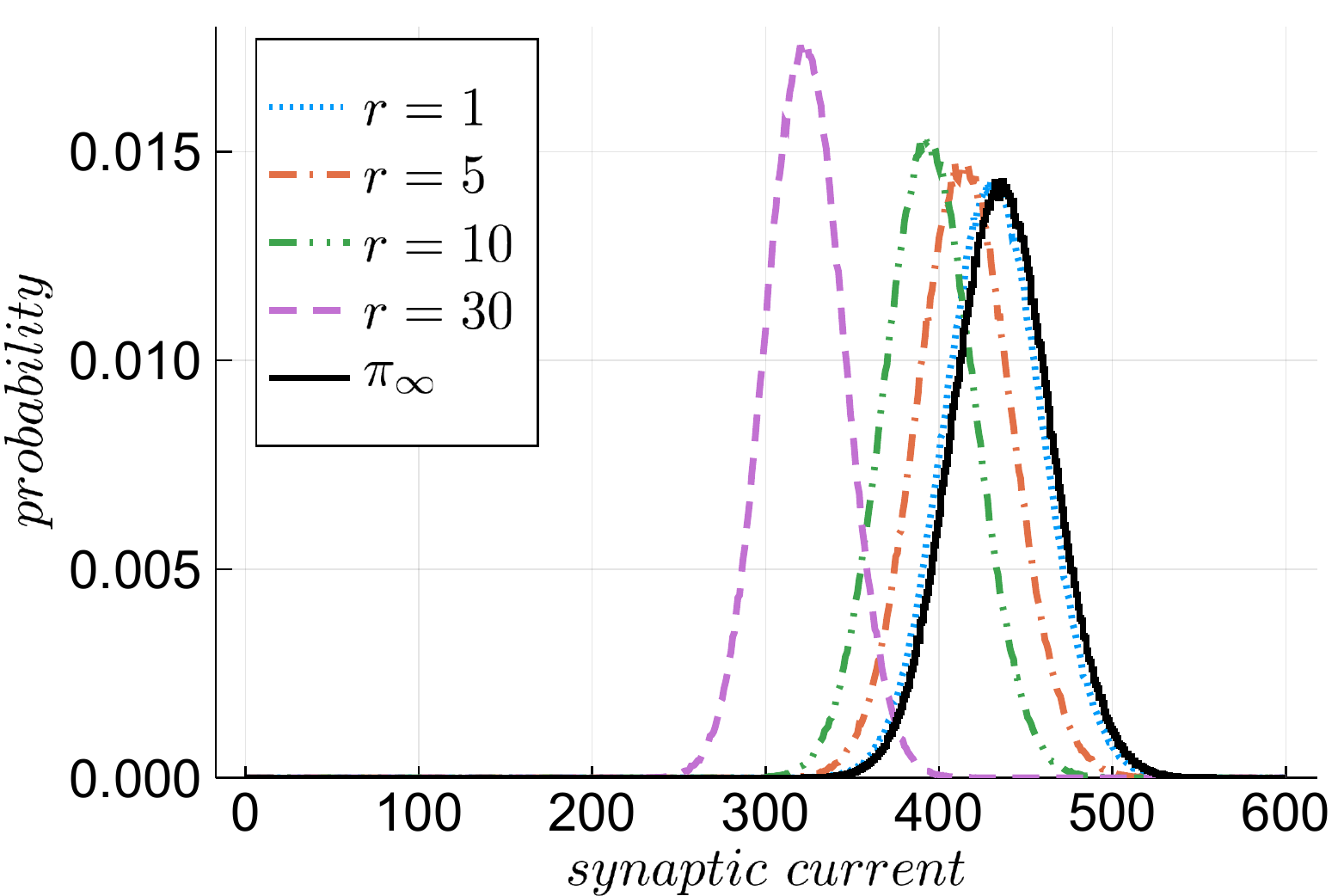} \label{sub:oubli-apprentissage-h0} }
		\subfloat[]{ \includegraphics[width=0.45\textwidth]{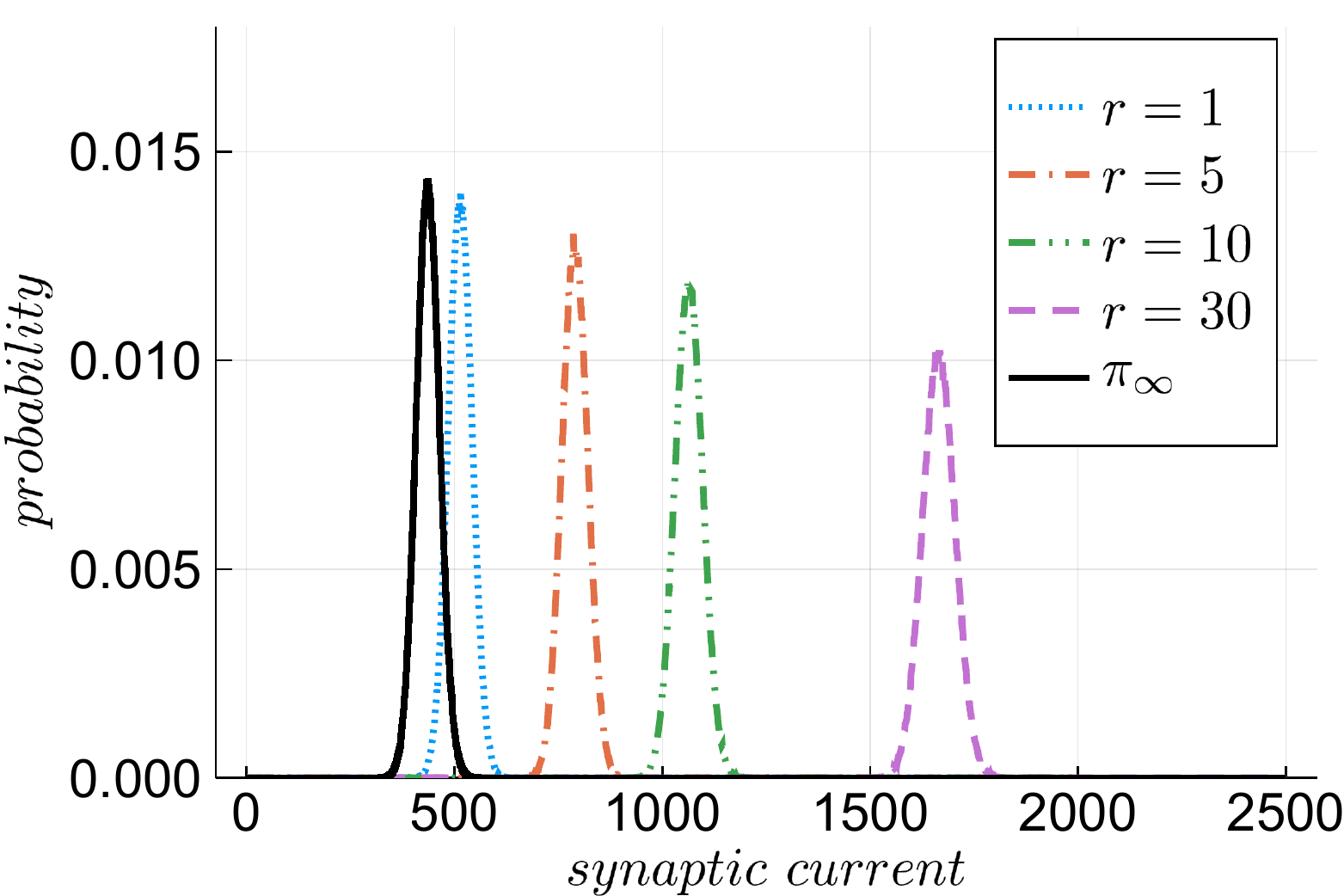} \label{sub:oubli-apprentissage} }
		\\ 
		\subfloat[]{ \includegraphics[width=0.45\textwidth]{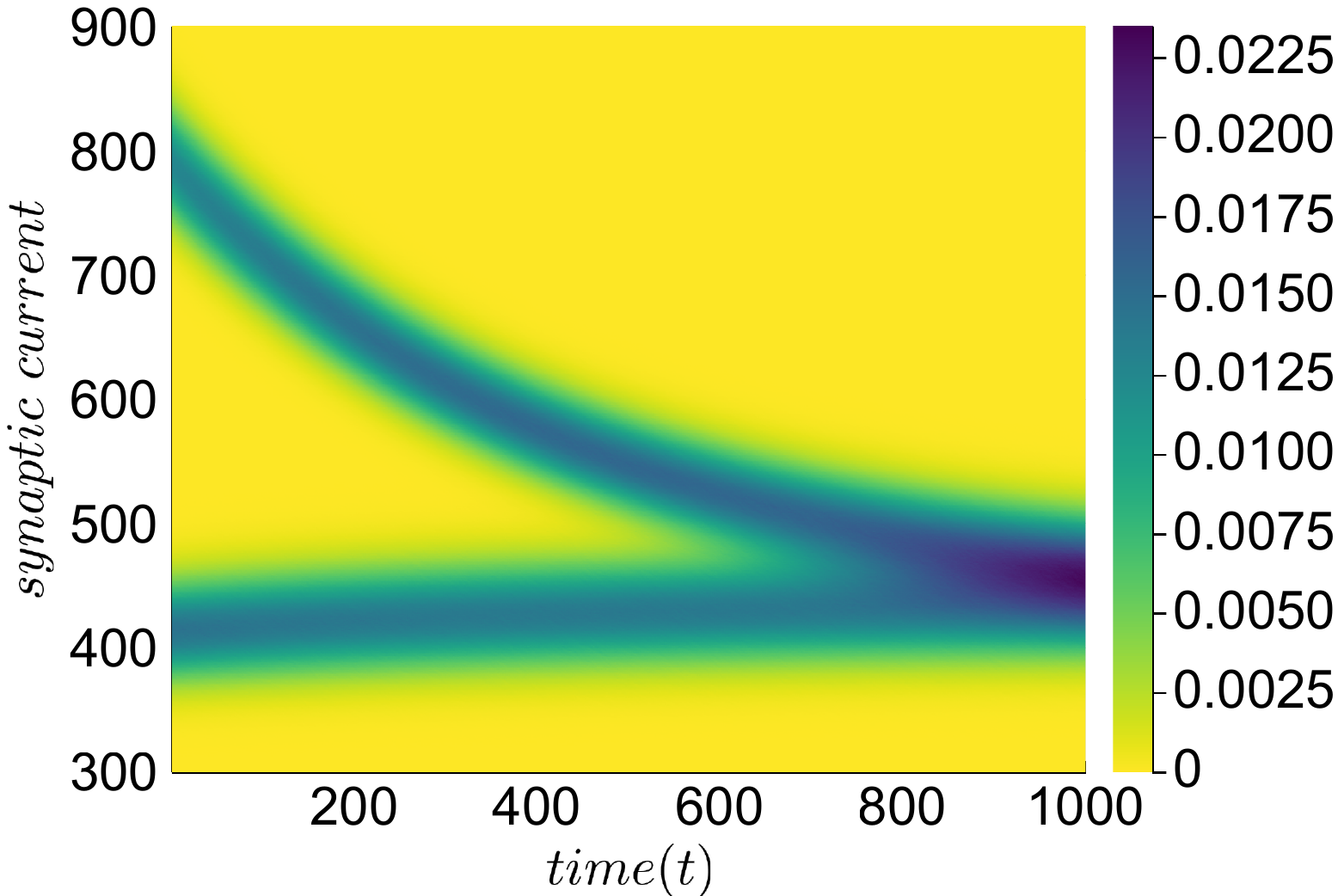}
			\label{sub:oubli-apprentissage-heatmap-r5} } 
		\subfloat[]{
			\includegraphics[width=0.45\textwidth]{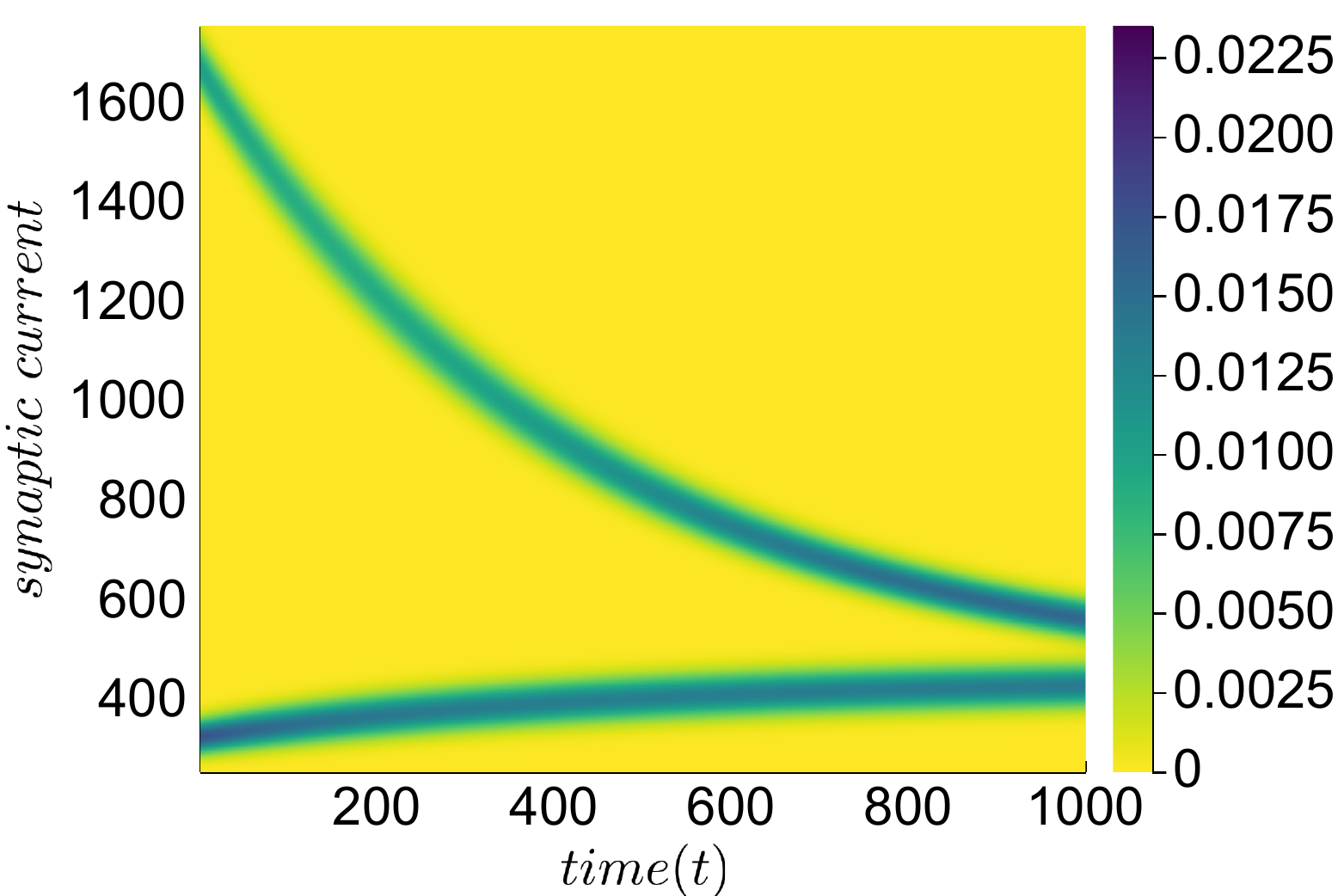} \label{sub:oubli-apprentissage-heatmap-r30} }
		\caption{
			\eqref{sub:oubli-apprentissage-h0}
			Distributions of	$h_1^0$, just after learning, 
			for different values of $r$ and the invariant measure.
			\eqref{sub:oubli-apprentissage}
			Distributions of $h_1^1$, for different values of $r$ and the invariant measure.
			\eqref{sub:oubli-apprentissage-heatmap-r5}
			The sum of the distributions of $h_t^0$ and $h_t^1$ for $r=5$.
			\eqref{sub:oubli-apprentissage-heatmap-r30}
			The sum of the distributions of $h_t^0$ and $h_t^1$ for $r=30$.
			The colour bar gives the probability values.
			Parameters: $N_{MC} =10^6,\ N = 20\ 000,\ f_N = 0.1,\ 
			q^-_{01} = q^-_{10} = 0.01 \text{ and } q^+= 0.05.$
		}\label{image.oubli-apprentissage} 
	\end{figure}
	Because of the parameters choice, the distributions of $h_1^0$ are close to the invariant measure
	$\pi_{\infty}$ whereas the distributions of $h_1^1$ are further from it,
	see Figures~\ref{sub:oubli-apprentissage-h0} and ~\ref{sub:oubli-apprentissage}.
	Moreover, the forgetting is really slow.
	However, if we want the signal to be learnt correctly with such a small $q^+$,
	then $r$ has to be high enough.
	This shows the need of a large $r$ in view of a	slow forgetting.
	Figures~\ref{sub:oubli-apprentissage-heatmap-r5} and~\ref{sub:oubli-apprentissage-heatmap-r30}
	show well the difference brought by a higher value of $r$:
	the separation between the two distributions is clearer.
	
	\section{Discussion}\label{sec:discussion}
	We provide a mathematical framework to study the memory retention of random signals by a
	recurrent neural network with binary neurons and binary synapses.
	We thus consider a paradigm linking synaptic plasticity and memory:
	a stimulus is remembered as long as its trace in the synaptic weights is strong enough.
	In order to measure the memory of a stimulus, we study the synaptic current onto one neuron
	during the presentation of this stimulus.
	First, we compute the spectrum of the transition matrix of the Markov chain associated to
	the synaptic current.
	This enables us to conclude that the eigenvalues are strictly different whatever the parameters are.
	In particular, we can compute the rate of convergence of the chain to its invariant measure,
	see Corollary~\ref{cor:conv-inv-binmix}.
	Then, we carry on the work done by~\cite{amit_precise_2010} on the dynamics of the
	distributions of the synaptic current and their invariant distribution.
	This leads us to control the form of these distributions.
	Their properties give enough information to find a lower bound on the time a neuron keeps a
	good estimate on its response to the first stimulus and hence remembers it. 
	We measure the quality of this estimation by performing a statistical test based on the
	observation of the synaptic current onto one neuron. 
	We define an error associated to this test which depends on two distributions:
	the distribution of the synaptic current knowing that the neuron was selective to
	the initial signal and the distribution knowing that the neuron was not selective.
	Finally, unlike previous studies, we take into account the possibility that heterosynaptic
	and homosynaptic depressions scale differently in the network size $N$
	and we consider the role of presenting several times a signal in the learning phase.
	
	We use the model presented by~\cite{amit_learning_1994} because of its relative simplicity and
	its consideration of synapse correlations.
	Their study focused on the first two moments of the synaptic current.
	It leads to a result on the memory capacity of the network which depends on a global variable,
	the so-called signal-to-noise ratio (SNR).
	In particular, they studied the SNR in the large $N$ asymptotic.
	They obtained a large SNR when the coding level $f_N$ is low and the depression probabilities are
	proportional to $f_N$: $q_{01,N}^- \propto q^+ f_N$ and $q_{10,N}^- \propto q^+ f_N$.
	The lowest coding level possible $f_N$ is on the order of $\frac{\log(N)}{N}$ and it gives a memory
	capacity on the order of $\frac{-1}{\log(\lambda_1)}\sim_{N\infty}\frac{1}{f_N^2}$.
	In~\cite{romani_optimizing_2008,amit_precise_2010}, they assumed that $q_{10,N}^- = 0$
	and showed the same result using a Gaussian approximation of the synaptic currents.
	Under the same assumption as in~\cite{amit_learning_1994}
	$\left(q_{01,N}^-,\ q_{10,N}^-\propto q^+ f_N \text{ and } f_N \rightarrow 0\right)$,
	our result also predicts a forgetting time on the order of 
	$\frac{-1}{\log(\Lambda_0)}\sim_{N\infty}\frac{1}{f_N^2}$, see Theorem~\ref{theo:main_result_q_N}.
	Moreover, we give a result for depression probabilities not depending on $N$ and our result link
	the probability of error to the parameters. 
	Note the presence of $\Lambda_0$ in our result rather than $\lambda_1$ in previous studies.
	This difference comes from our different measure of memory lifetime.
	The SNR analysis is based on the convergence of the means of the synaptic currents
	whereas our retrieval criterion requests the knowledge of their entire distributions.
	Indeed, we search for a memory lifetime obtained with a control on the errors \(p_e^0\) and \(p_e^1\).
	We conjecture that we could prove similar result as ours with $\lambda_1$ rather than
	$\Lambda_0$.
	Finally, our results do not necessarily need the large $N$ asymptotic.
	Nevertheless, in this asymptotic, the expression of $\hat{t}$ simplifies, see
	Remark~\ref{rem:main-result1}.
	
	In this study, we assume that learning is generated by the divergence of the distributions
	of the synaptic currents $h_{t}^0$ and $h_{t}^1$ from their invariant distribution,
	see Figure~\ref{fig:image-discussion}.
	\begin{figure}[h] 
		\centering 
		\includegraphics[width=0.9\textwidth]{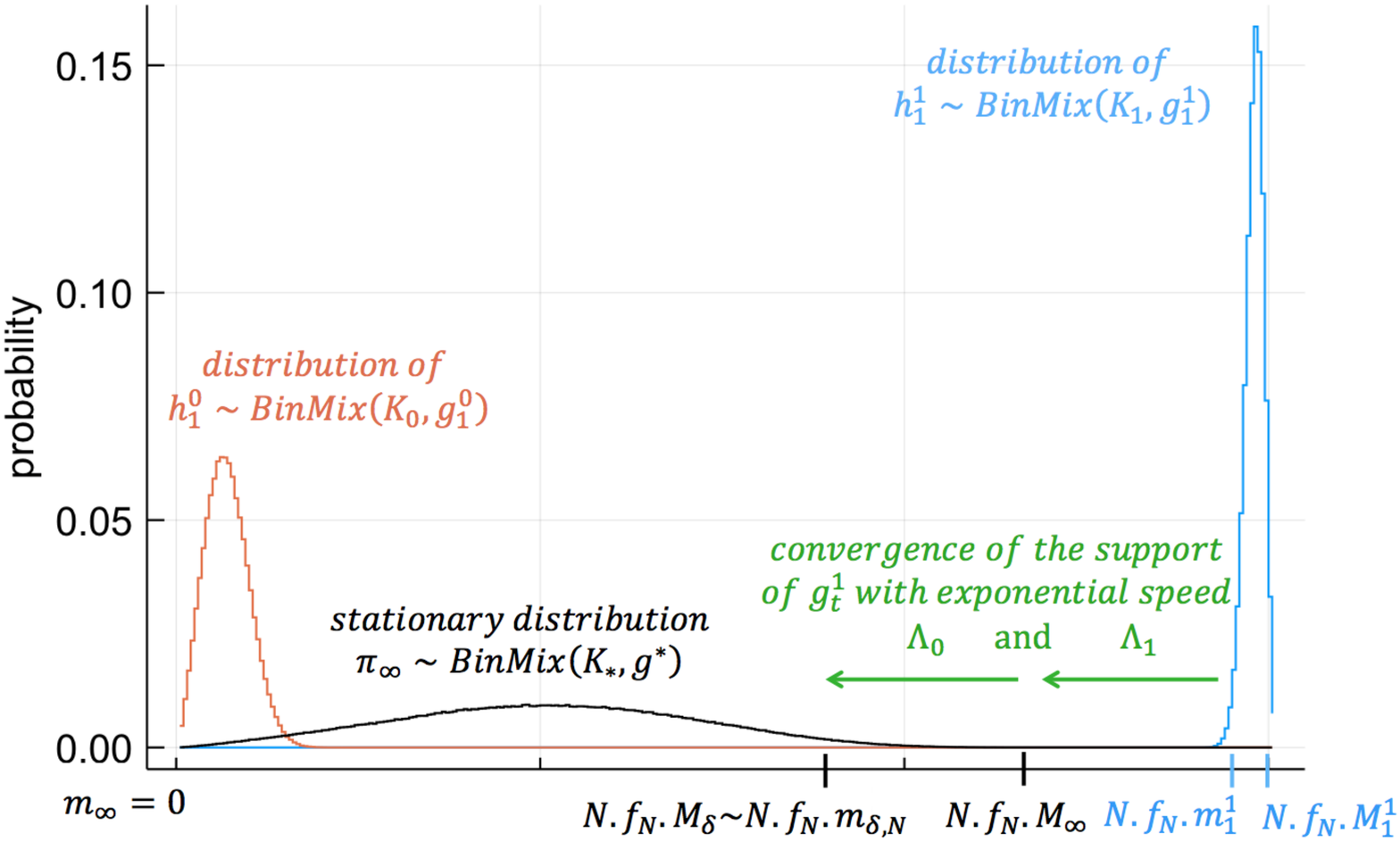}
		\caption{
			Illustration of the notations. The variables $K_0,\ K_1,\ K_*$ have Binomial laws with
			parameters \(N\) and \(f_N\).
			They are respectively independent from $h_1^0,\ h_1^1,\ \pi_\infty$.
		}\label{fig:image-discussion} 
	\end{figure}
	\noindent The main role of the number of signal presentations ($r$)
	is to separate these two distributions.
	Indeed, the larger the $r$, the more separated the support of the mixing distributions
	$g_1^0$ and $g_1^1$ are.
	In our proofs, we compare $g_t^1$ to $g^*$ after showing that as long as $g_t^1$ is far enough
	from $g^*$ it is far enough from $g_t^0$, see Lemma~\ref{lem:mass_move_right}.
	As a consequence, the expression of $\hat{t}$ is an increasing function of $m_1^1-M_\infty$,
	and so of $r$.
	
	Let us now discuss the roles of the coding level, the potentiation and depression probabilities.
	They affect both learning and forgetting.
	The coding level directly affects the number of synapses candidate to depression and potentiation.
	Indeed, looking at an individual synapse, its probability to potentiate is $f_N^2q^+$ and
	its probability to depress is $f_N(1-f_N)(q_{10}^- + q_{01}^-)$.
	Thus, when the coding level is close to one, the fluctuations are important
	and seem to cause a fast forgetting as shown in the illustrations of Section~\ref{sec-2}.
	Therefore, we used a low coding level, see Assumption~\ref{ass-main:on-f}.
	This choice slows down the forgetting.
	However, $f_N$ cannot be too small because it is detrimental to the learning phase
	as the distance between the two conditional distributions depends on $f_N$.
	More particularly, it depends on $Nf_N$ which then need to be large enough,
	see Assumption~\ref{ass-main:on-f}.
	The last parameters we can tune are the potentiation and depression probabilities.
	As for $f_N$, there is a compromise between their role in learning and forgetting.
	Indeed, in order to promote learning, they need to be close enough to one
	but on the contrary, small probabilities reduce the forgetting rate.
	So we propose to take a potentiation probability ($q^+$) on the order of $1$,
	to learn quickly, and small depression probabilities, to forget slowly.
	Potentiation increases the synaptic currents so it leads to a shift of the distribution of $h_t^1$
	to the right and for the same reasons, depression results in a shift of the distribution of
	$h_t^0$ to the left.
	Therefore, smaller depression than potentiation implies that the distribution of $h_t^1$ is
	significantly shifted to the right whereas the distribution of $h_t^0$ is slightly shifted to the left.
	In view of learning, the initial separation between distributions can be limited if the invariant
	distribution $\pi_{\infty}$ is already concentrated on high values of synaptic currents. 
	As there are two depression probabilities, this situation can be avoided by choosing
	one probability big enough and the other one smaller.
	For example, when depression probabilities depend on $N$ under Assumption~\ref{ass-main:on-q10-q01},
	both $q_{10,N}^-$ and $q_{01,N}^-$ converge to $0$.
	If they both converge too fast ($a_N$ and $b_N$ converge to $0$),
	the invariant measure is concentrated around one and no learning is possible.
	However, if either $a_N$ or $b_N$ does not converge to $0$, then the invariant measure is 
	not concentrated around $0$ and learning is possible. Then, depending on the different
	large $N$ asymptotic of $a_N$ and $b_N$, we computed the different memory lifetime summarized
	in Table~\ref{fig:table_asympt}. The best memory lifetimes are on the order of $\frac{1}{f_N^2}$
	and are obtained when $a_N$ (resp. $b_N$)
	converges to $0$ and $b_N$ (resp. $a_N$) converges to a constant in $\reels^+$ (resp. $\reels_*^+$)
	or $(a_N,b_N)$ converges to constants in $\reels_*^+ \times \reels^+$.
	Thus, if one wants to increase the memory lifetime beyond this order,
	we seem to need a model more complex. 
	
	Our study is valid for a classic learning, which needs multiple stimulus presentations,
	but also for a one shot learning.
	This last one is possible only with a specific choice of parameters.
	Indeed, when presenting a stimulus, the synaptic weights between selective neurons
	need to be potentiated with a high probability (high $q^+$).
	When presenting other stimuli, these same weights need to have a very small probability of
	undergoing depression (low $q^-_{01}$ and $q^-_{10}$ ).
	As a result, following the presentation of a stimulus, selective neurons develop strong
	links and then these connections take time to disappear.
	Thus, the experiment associated with this model would focus on recognition memory.
	A well-known experiment in this field was carried out by \cite{standing_learning_1973}.
	He showed that humans are able to recognize up to 10,000 images,
	presented only once, with 90 percent success rate.
	
	Many perspectives can be studied as a follow-up to this study.  
	First, the analysis carried out on the synaptic current onto one neuron could
	be extended to the entire vector of synaptic currents.
	The correlations between synaptic weights would then play a major role.
	In addition, the model could be completed in order to get closer to biology.
	Indeed, the formation of synaptic memory is far more complex than in our model.
	In particular, the link between the dynamics of the neurons and the synaptic weight is missing. Improving the model in this direction could be done by considering more structured and complex
	external signals, adding neural layers and a more realistic membrane potential neural dynamics.
	In the literature, adding synaptic states does not seem to be successful as the authors stated
	in~\cite{fusi_limits_2007,huang_capacity_2011}, whereas meta-plastic transitions brought better SNR
	results~\cite{fusi_cascade_2005,roxin_efficient_2013,benna_computational_2016}.
	Adding neural dynamics in such models would be a next challenging step.
	Nevertheless, the model analysed here illustrates well the trade-off between the
	plastic and the stable characteristics of memory.
	Indeed, learning implies changes of synaptic weights (plasticity) as well as mechanisms
	which maintain them (stability).
	In mathematical terms, stability is related to the minimal convergence rate and
	plasticity refers to the sensibility to disturbance.
	We see that there is a compromise: the more a dynamics is sensitive to
	disturbances, the less it is stable and vice-versa.
	
	\section*{Appendix} \appendix \section{Proofs}\label{app:proofs}
	\subsection{Proof of Proposition~\ref{prop:dyn-h}}\label{app:prop-dyn-h}
	\begin{nota}\label{not:G-ab}
		Let $Z$ be a random variable in $[0,1]$ with distribution $g_Z$ and
		cumulative distribution function $G_Z$.
		We denote by $g_{Z,(a,b)}\in \mathcal P([0,1])$ the distribution such that 
		\[
		\forall u \in \reels,\qquad
		G_{Z,(a,b)}(u) =
		G_Z\left(\frac{u-b}{a-b}\right).
		\]
	\end{nota} 
	Proposition~\ref{prop:dyn-h} relies on the following
	\begin{lem}\label{lem-still-BinMix} 
		Let $Z$ be a mixture of Binomial $Z =  \BinMix(K,Y_Z)$. Let $0 \leq b < a < 1$.
		Conditionally on \(Z\), consider two independent Binomial distributions \(\Bin(Z,a)\)
		and \(\Bin(K-Z,b)\) and define \(X = \Bin(Z,a) + \Bin(K-Z,b)\).
		Then
		\begin{equation}\label{eq:20191007_1}
		X\overset{\mathcal{L}}{=} \BinMix\left(K,Y_X\right) \ \text{  with  }\ \ Y_X = (a-b)Y_Z+b.
		\end{equation}
		In particular, $G_X(u) = G_{Z,(a,b)}(u).$
	\end{lem}
	\begin{proof}
		Let $\tilde{U}$, $(U_i)_{1\leq i \leq K}$, $(\xi_i)_{1\leq i \leq K}$,
		$(\eta_i)_{1\leq i \leq K}$ and $(W_i)_{1\leq i \leq K}$ be
		\textit{i.i.d.} random variables following the uniform law on $[0,1]$. 
		By the first point of Remark~\ref{rem-equivalence-BinMix-Bin}, $Z$ is the sum of
		$(Z_i)_{1\leq i \leq K}$ \textit{i.i.d.} Bernoulli of parameter $Y_Z=G_Z^{-1}(\tilde{U})$.
		Thus, we obtain that conditionally on $Z$,
		\[
		X = \underbrace{\sum_{i=1}^{K}Z_i\mathbbm{1}_{\left\{ \xi_i\leq
				a\right\}}}_{\overset{\mathcal{L}}{=}  \Bin(Z,a)} +
		\underbrace{\sum_{i=1}^{K}(1-Z_i)\mathbbm{1}_{\left\{ \eta_i\leq
				b\right\}}}_{\overset{\mathcal{L}}{=}   \Bin(K-Z,b)}
		\]
		where the Binomials are independent.
		Then, let consider 
		\(\forall i, Z_i= \mathbbm{1}_{\left\{U_i\leq G_Z^{-1}(\tilde{U})\right\}}\).
		Thus,
		\begin{align}
		\nonumber X &= \sum_{i=1}^{K}\mathbbm{1}_{\left\{U_i\leq
			G_Z^{-1}(\tilde{U})\right\}}\mathbbm{1}_{\left\{ \xi_i\leq
			a\right\}}+\sum_{i=1}^{K}\mathbbm{1}_{\left\{U_i>
			G_Z^{-1}(\tilde{U})\right\}}\mathbbm{1}_{\left\{ \eta_i\leq b\right\}}.\\
		\text{	So }\quad
		\label{eq:couple-unif}
		X & = \sum_{i=1}^{K}\mathbbm{1}_{\left\{U_i\leq G_Z^{-1}(\tilde{U}), \xi_i\leq
			a\right\} \bigcup \left\{U_i> G_Z^{-1}(\tilde{U}), \eta_i\leq b\right\}}.
		\end{align}
		\begin{figure}[h]
			\centering
			\begin{tikzpicture}
			\fill[color=gray!20] (1,0) -- (1,1.5) -- (4,1.5)  -- (4,0) -- cycle;
			\fill[color=gray!100] (0,0) -- (1.5,0) -- (1.5,3)  -- (0,3) -- cycle;
			\draw[->] (0,0) -- (4.5,0);
			\draw (4.5,0) node[right] {$U_i$};
			\draw [->] (0,0) -- (0,4.5);
			\draw (0,4.5) node[above] {$\xi_i, \eta_i$};
			\draw (0,1.5) node[left] {$b$};
			\draw (0,3) node[left] {$a$};
			\draw (1.5,0) node[below] {$G_Z^{-1}(\tilde{U})$};
			\draw (4,0) node[below] {$1$};
			\draw (0,4) node[left] {$1$};
			\draw (-0.1,0) node[below] {$0$};
			\draw [dashed] (0,4) -- (4,4);
			\draw [dashed] (4,0) -- (4,4);
			\end{tikzpicture}
			\caption{In gray, the domain to which the couple $\left(U_i,\xi_i,\eta_i\right)$
				needs to belong to from the equation~\eqref{eq:couple-unif}.
			}\label{fig:couple-unif}
		\end{figure}
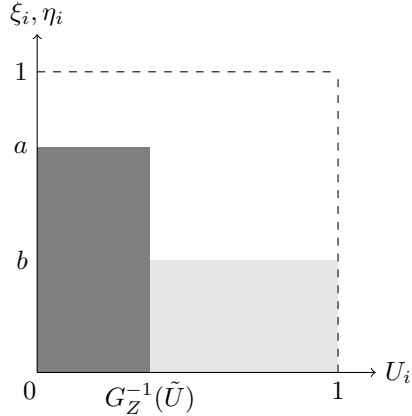
		
		For all Borel set $\mathcal{D}\subset [0,1]^3$,
		$\proba\left(\left(U_i,\xi_i, \eta_i\right)\in \mathcal{D}\right) =
		\mathcal{V}\left(\mathcal{D}\right)$ where $\mathcal{V}\left(\mathcal{D}\right)$
		is the volume of $\mathcal{D}$.
		Thus, let $W_i \overset{\mathcal{L}}{=} \mathcal{U}\left([0,1]\right)$, then 
		$\proba\left(\left(U_i,\xi_i,\eta_i\right)\in \mathcal{D}\right)
		= \proba\left(W_i \leq \mathcal{V}\left(\mathcal{D}\right)\right).$
		We put $\xi_i$ and $\eta_i$ on the same axis as they do not depend one on the other so
		that the volume
		$\mathcal{V} \left( \left\{ U_i \leq G_Z^{-1}( \tilde{U} ), \xi_i \leq a \right\} \bigcup
		\left\{ U_i > G_Z^{-1}( \tilde{U} ), \eta_i \leq b \right\} \right)$
		is equal to the sum of the tow grey areas (see Figure~\ref{fig:couple-unif}).
		We deduce that
		\[	
		X = \sum_{i=1}^{K}\mathbbm{1}_{\left\{ W_i\leq b+(a-b)G_Z^{-1}(\tilde{U})\right\}}
		= \sum_{i=1}^{K}\mathbbm{1}_{\left\{ G_Z\left(\frac{W_i-b}{a-b}\right)\leq
			\tilde{U}\right\}}
		= \sum_{i=1}^{K}\mathbbm{1}_{\left\{G_X(W_i)\leq \tilde{U}\right\}},
		\]
		\noindent with \(G_X(w) = G_{Z,(a,b)}(w)\).
		We conclude that \eqref{eq:20191007_1} is satisfied.
	\end{proof}
	Proof of Proposition~\ref{prop:dyn-h}.
	\begin{proof}
		We first show~\eqref{eq:recursive-h-t0} and~\eqref{eq-cdf-H-1} for $h_{t,K}$,
		then the rest follows.
		
		At $t=1$, from equation~\eqref{def-h-t0-Bin} we get
		\begin{align*}
		\mathcal{L}\left(h_{1,K}|\xi^1_{0}=1,h_{-r+1,K}\right) &= \Bin\left(h_{-r+1,K},1\right) + \Bin\left(K-h_{-r+1,K},
		1-(1-q^+)^r\right)\\ 
		\mathcal{L}\left(h_{1,K}|\xi^1_{0}=0,h_{-r+1,K}\right) &= \Bin\left(h_{-r+1,K}, (1-q_{01}^-)^r\right).
		\end{align*}
		\noindent Applying twice Lemma~\ref{lem-still-BinMix} with \((a,b) = (1,1-(1-q^+)^r)\) and
		then \((a,b) = (1-(1-q_{01}^-)^r,0)\), we obtain using notation~\ref{not:G-ab}
		\begin{align*}
		\mathcal{L}\left(h_{1,K}|\xi^1_{0}=1,h_{-r+1,K}\right)	& \overset{\mathcal{L}}{=}  
		\BinMix\left(K,g_{-r+1,(1,1-(1-q^+)^r)}\right) \\ 
		\mathcal{L}\left(h_{1,K}|\xi^1_{0}=0,h_{-r+1,K}\right) &
		\overset{\mathcal{L}}{=}   \BinMix\left(K,g_{-r+1,(1-(1-q_{01}^-)^r,0)}\right).
		\end{align*}
		\noindent Thus,
		
		\noindent\( 
		\proba\left(h_{1,K}=j|h_{-r+1,K}\right)\)\\
		\noindent\(\begin{aligned}
		=&\proba\left(\xi^1_{0}=1\right)\proba(h_{1,K}=j|\xi^1_{0}=1,h_{-r+1,K})
		+
		\proba\left(\xi^1_{0}=0\right)\proba(h_{1,K}=j|\xi^1_{0}=0,h_{-r+1,K})\\ 
		=&
		\displaystyle \binom{K}{j}\int_{0}^{1}u^j(1-u)^{K-j}\big(f_N g_{-r+1,(1,1-(1-q^+)^r)}(du)
		+
		(1-f_N)g_{-r+1,(1-(1-q_{01}^-)^r,0)}(du)\big),
		\end{aligned}\)
		\\
		
		which enables to get~\eqref{eq:recursive-h-t0}.
		
		\noindent Now, assume that $h_{t,K}\overset{\mathcal{L}}{=}\BinMix(K,g_t)$,
		for some fixed \(t\geq 1\).
		Then, by equation~\eqref{def-h-Bin}
		\begin{align*} 
		\mathcal{L}\left(h_{t+1,K}|\xi^1_{t}=1,\ h_{t,K}\right) &= \Bin\left(K-h_{t,K},f_N q^+\right) + \Bin\left(h_{t,K},1-(1 - f_N) q_{10}^-\right), \\ \mathcal{L}\left(h_{t+1,K}|\xi^1_{t}=0,\ h_{t,K}\right) &= \Bin\left(h_{t,K},1-f_N q_{01}^-\right), 
		\end{align*} 
		where Binomials are independent conditionally on $h_{t,K}$.
		Applying twice Lemma~\ref{lem-still-BinMix} with \((a,b) = (1-(1 -
		f_N) q_{10}^-,f_N q^+)\) and \((a,b) = (1-f_Nq_{01}^-,0)\), we get
		\begin{align*} 
		\mathcal{L}\left(h_{t+1,K}|\xi^1_{t}=1\right)	& \overset{\mathcal{L}}{=}  
		\BinMix\left(K,g_{t,(1-(1 - f_N) q_{10}^-,f_N q^+)}\right) \\ 
		\mathcal{L}\left(h_{t+1,K}|\xi^1_{t}=0\right) &
		\overset{\mathcal{L}}{=}   \BinMix\left(K,g_{t,(1-f_Nq_{01}^-,0)}\right).
		\end{align*} 
		Hence, $h_{t+1,K}\overset{\mathcal{L}}{=}   \BinMix\left(K,f_N g_{t,(1-(1 - f_N) q_{10}^-,f_Nq^+)} + (1-f_N) g_{t,(1-f_Nq_{01}^-,0)}\right)$, and we deduce that 
		$h_{t+1,K}\overset{\mathcal{L}}{=}   \BinMix(K,g_{t+1})$ with $G_{t+1}(x)
		= \mathcal{R}(G_t)(x)$.
		
		For the processes $\left(h_{t,K}^y\right)_{t\geq 0}$, we proceed exactly with the 
		same method with the fact that $\xi_0^1 = y$ in Proposition~\ref{prop-Markov}.
	\end{proof}
	
	\subsection{Proof of Proposition~\ref{prop-R-contracting}}\label{app-R-contracting}
	\begin{proof}
		\textbf{1. The map \(\mathcal{R}\) is a contraction}
		
		\noindent Let
		$\Gamma_1,\Gamma_2\in F_{[0,1]}$. We recall that $\Lambda_1 = 1-(1-f_N)q_{10}^--f_Nq^+$,
		$\Lambda_0 = 1-f_Nq_{01}^-$.
		\noindent$\begin{aligned} 
		&\|\mathcal{R}(\Gamma_2) - \mathcal{R}(\Gamma_1)\|_{L^1(0,1)}\\
		&\leq \int_{0}^{1}f_N\left| \Gamma_2\left(\frac{u-f_N q^+}{\Lambda_1}\right)
		-\Gamma_1\left(\frac{u-f_N q^+}{\Lambda_1}\right)\right| 
		+(1-f_N)\left|
		\Gamma_2\left(\frac{u}{\Lambda_0}\right)-\Gamma_1\left(\frac{u}{\Lambda_0}\right)\right| du\\
		&=f_N \int_{f_N q^+}^{1-(1 - f_N) q_{10}^-}\left| \Gamma_2\left(\frac{(u-f_N q^+) }{\Lambda_1}\right)-\Gamma_1\left(\frac{(u-f_N q^+) }{\Lambda_1}\right)\right| du\\ 
		& \qquad \qquad + (1-f_N)\int_{0}^{\Lambda_0}\left| \Gamma_2\left(\frac{u}{\Lambda_0}\right)-\Gamma_1\left(\frac{u}{\Lambda_0}\right)\right| du
		\end{aligned}$\\
		\noindent$\begin{aligned}
		&=f_N \Lambda_1 \int_{0}^{1}\left| \Gamma_2\left(u\right)-\Gamma_1\left(u\right)\right| du
		+ (1-f_N)\Lambda_0\int_{0}^{1}\left| \Gamma_2\left(u\right)-\Gamma_1\left(u\right)\right| du\\ 
		&= \underbrace{\left(f_N \Lambda_1 + (1-f_N)\Lambda_0\right)}_{\lambda_1}\|\Gamma_2-\Gamma_1\|_{L^1(0,1)}. 
		\end{aligned}$\\
		As $\lambda_1 <1$, the map $\mathcal{R}$ acting on $F_{[0,1]}$ is strictly contracting
		in $L^1(0,1)$.
		
		\noindent	\textbf{2. Existence and uniqueness of a fixed point}
		
		\noindent	We now prove the second point of the Lemma.
		For all $\Gamma_0\in F_{[0,1]}$, by contraction of $\mathcal{R}$,
		$\left(\mathcal{R}^n\left(\Gamma_0\right)\right)_{n\geq 0}$ is a Cauchy sequence
		for the $L^1(0,1)$ norm.
		By completeness of $L^1(0,1)$, this sequence converges to some $\Gamma \in L^1(0,1)$. 
		It remains to prove that \(\Gamma\) can be chosen in \(F_{[0,1]}\).
		First, any limit $\Gamma$ is non decreasing almost everywhere.
		Define $G^*(x) = \lim\limits_{y \rightarrow x_+}\Gamma(y)$.
		The function \(G^*\) is c\`adl\`ag and satisfies for every \(x \leq 0\), \(G^*(x)=0\) and
		for every \(x \geq 1\), \(G^*(x)=1\). 
		Thus $G^*\in F_{[0,1]}$ and $\mathcal{R}(G^*)=G^*$.
		Finally, the uniqueness of $G^*$ is deduced from the fact that $\mathcal{R}$ is strictly contracting.
	\end{proof} 
	\subsection{Proof of Lemma~\ref{lem:tail-bin}}\label{app-proof-lemmas}
	\begin{proof}
		We use the method of \cite{chernoff_measure_1952}.
		Let $S_N$ be the sum of $X_1, X_2, \cdots , X_N$ which
		are independent Bernoulli random variables of parameter $p$.
		
		\noindent For all $\varepsilon \in (0,1), u\in \reels^+$,
		
		\noindent$\begin{aligned}
		\proba\left(S_N \geq Np(1+\varepsilon)\right)&
		= 
		\proba\left(\e^{uS_N} \geq \e^{Np(1+\varepsilon)u}\right)
		\leq 
		\frac{\esp\left(\e^{uS_N}\right)}{\e^{Np(1+\varepsilon)u}}
		= 
		\frac{\prod_{i=1}^{N}\esp\left(\e^{uX_i}\right)}{\e^{Np(1+\varepsilon)u}} 
		\\
		&\leq
		\frac{\left(1+p(\e^u-1)\right)^N}{\e^{Np(1+\varepsilon)u}}
		\leq 
		\frac{\e^{Np(\e^u-1)}}{\e^{Np(1+\varepsilon)u}} = \e^{Np(\e^u-1-(1+\varepsilon)u)}.
		\end{aligned}$
		\\
		
		\noindent The minimum of the last term is reached for $u=\log(1+\delta)$ so
		\[
		\proba\left(X>Np(1+\varepsilon)\right) 
		\leq
		\left(\frac{\e^{\varepsilon}}{(1+\varepsilon)^{1+\varepsilon}}\right)^{Np}
		=
		\exp\left( Np\big( \varepsilon - (1+\varepsilon)\log(1+\varepsilon)\big) \right).
		\]
		From the inequality, $\forall z>0,\ \log(1+z)\geq \frac{2z}{2+z}$,
		we obtain~\eqref{eq:ineq-bin-pos}. In order to show~\eqref{eq:ineq-bin-neg},
		we proceed with the same method and use the inequality $\log(1+z)\geq \frac{z}{2}\frac{2+z}{1+z}$
		whenever $-1<z\leq0$.
	\end{proof}
	\section*{Acknowledgements}
	I am indebted to the help from Etienne Tanr\'e and Romain Veltz.
	
	\noindent I am very grateful to two anonymous referees for valuable comments and suggestions which really helped improving the paper.
	
	\noindent This research has received funding from the European Union's Horizon 2020 Framework Programme
	for Research and Innovation under the Specific Grant Agreement No. 785907 (Human Brain Project SGA2). 
	\bibliographystyle{apalike}
	\bibliography{biblio-memory} 
\end{document}